\documentclass[11pt,reqno,twoside]{amsart}
\usepackage{amssymb,amsmath,amsthm,soul,color,paralist}
\usepackage{t1enc}
\usepackage[cp1250]{inputenc}
\usepackage{a4,indentfirst,latexsym}
\usepackage{graphics}
\usepackage{mathrsfs}
\usepackage{cite,enumitem,graphicx}
\usepackage[colorlinks=true,urlcolor=blue,
citecolor=red,linkcolor=blue,linktocpage,pdfpagelabels,
bookmarksnumbered,bookmarksopen]{hyperref}
\usepackage[english]{babel}
\usepackage[left=2.50cm,right=2.50cm,top=2.72cm,bottom=2.72cm]{geometry}
\usepackage[metapost]{mfpic}
\usepackage[hyperpageref]{backref}
\usepackage[colorinlistoftodos]{todonotes}
\usepackage[normalem]{ulem}

\makeatletter
\providecommand\@dotsep{5}
\def\listtodoname{List of Todos}
\def\listoftodos{\@starttoc{tdo}\listtodoname}
\makeatother

\numberwithin{equation}{section}

\newcommand{\R}{\mathbb{R}}
\newcommand{\N}{\mathcal{N}}

\newcommand{\C}{\mathbb{C}}

\newcommand{\h}{H^{1/2}_{\varepsilon}}

\newcommand{\ri}{\rightarrow}

\DeclareMathOperator{\supp}{supp}
\DeclareMathOperator{\e}{\varepsilon}

\newtheorem{prop}{Proposition}[section]
\newtheorem{lem}{Lemma}[section]
\newtheorem{thm}{Theorem}[section]

\newtheorem{cor}{Corollary}[section]

\newtheorem{remark}{Remark}[section]

\title[fractional Schr\"odinger equations with magnetic field]{On a fractional magnetic Schr\"odinger equation in $\R$ with exponential critical growth}

\author[V. Ambrosio]{Vincenzo Ambrosio}
\address{Vincenzo Ambrosio\hfill\break\indent
Dipartimento di Ingegneria Industriale e Scienze Matematiche \hfill\break\indent
Universit\`a Politecnica delle Marche\hfill\break\indent
Via Brecce Bianche, 12\hfill\break\indent
60131 Ancona (Italy)}
\email{vincenzo.ambrosio2@unina.it}


\begin{document}

\maketitle

\begin{abstract}
We deal with a class of fractional magnetic Schr\"odinger equations in the whole line with exponential critical growth. Under a local condition on the potential, we use penalization methods and Ljusternik-Schnirelmann category theory to investigate the existence, multiplicity and concentration of nontrivial weak solutions. 
\end{abstract}

\maketitle

\section{Introduction}

\noindent 
In this paper we are interested in the existence, multiplicity and concentration of nontrivial solutions for the following fractional nonlinear Schr\"odinger equation
\begin{equation}\label{P}
\e(-\Delta)_{A/\e}^{1/2}u+V(x)u=f(|u|^{2})u \quad \mbox{ in } \R,
\end{equation}
where $\e>0$ is a parameter,  $A:\R\ri \R$ is a magnetic field belonging to $C^{0,\alpha}(\R, \R)$ for some $\alpha\in (0, 1]$, $V: \R\rightarrow \R$ is an electric potential, $f:\R\ri \R$ is a continuous nonlinearity, and $(-\Delta)^{1/2}_{A}$ is the $\frac{1}{2}$-magnetic Laplacian which is defined as 
\begin{equation*}
(-\Delta)^{1/2}_{A}u(x)
:=\frac{1}{\pi}\lim_{r\rightarrow 0} \int_{\R\setminus (x-r, x+r)} \frac{u(x)-e^{\imath (x-y)\cdot A(\frac{x+y}{2})} u(y)}{|x-y|^{2}} dy, 
\end{equation*}
for any $u: \R\ri \C$ sufficiently smooth.
We recall that the fractional magnetic Laplacian $(-\Delta)^{s}_{A}$, with $s\in (0, 1)$, has been recently proposed in \cite{DS, I10} as an extension in the fractional setting of the well-known magnetic Laplacian $(\frac{1}{\imath}\nabla -A(x))^{2}$; see \cite{AFF, BF, EL, Cingolani, K} for more details.

More in general, nonlocal and fractional operators have received a considerable attention from many mathematicians, both for pure academic research and for their presence in many models coming from different fields, such as phase transitions, optimization, finance, minimal surfaces, just to mention a few; see \cite{DPV, MBRS} for a more exhaustive introduction on this subject. 

We note that if $A=0$, then $(-\Delta)^{s}_{A}$ reduces to the celebrated fractional Laplacian $(-\Delta)^{s}$, and \eqref{P} can be regarded as a particular case of the well-known fractional Schr\"odinger equation 
\begin{align}\label{FSE}
\e^{2s}(-\Delta)^{s}u+V(x) u=f(x, u) \quad \mbox{ in } \R^{N},
\end{align}
which has received a tremendous popularity in the last years due to its crucial role in the study of fractional quantum mechanics; see \cite{Laskin1}. In fact, several existence and multiplicity results have been established for \eqref{FSE} under different conditions on the potential $V$ and the nonlinearity $f$ applying appropriate variational and topological arguments; see \cite{AM, A1, DDPW, DPMV, FQT, Is} and references therein.

However, when $N=1$ and $s=\frac{1}{2}$ in \eqref{FSE}, only few papers appear in literature; see \cite{ADOM, Cabre-Sola, dSA, DOMS, IS, FLS}. Indeed, one of the main difficulty in the study of this class of problems is related to the fact that the embedding $H^{1/2}(\R, \R)\subset L^{q}(\R, \R)$ is continuous for all $q\in [2, \infty)$ but not in $L^{\infty}(\R, \R)$; see \cite{DPV}. This means that the maximal growth that we have to impose on the nonlinearity $f$ to deal with \eqref{FSE} via variational methods in a suitable subspace of  $H^{1/2}(\R, \R)$, is given by $e^{\alpha_{0}|u|^{2}}$ as $|u|\ri \infty$ for some $\alpha_{0}>0$. More precisely, this fact is a consequence of the following fractional Moser-Trudinger inequality established by Ozawa in \cite{Ozawa}:
\begin{thm}\label{MT}
There exists $\omega \in (0, \pi]$ such that, for all $\alpha \in (0, \omega)$ there exists $C_{\alpha}>0$ such that 
$$
\int_{\R} (e^{\alpha|u|^{2}}-1) dx\leq C_{\alpha} \|u\|^{2}_{L^{2}(\R)}, 
$$
for all $u\in H^{1/2}(\R, \R)$ with $\|(-\Delta)^{1/4}u\|^{2}_{L^{2}(\R)}\leq 1$. 
\end{thm}
It is worth pointing out that  Moser-Trudinger type inequalities \cite{MT, Trudinger} have been widely used in the study of several bidimensional elliptic problems with exponential critical growth; see for instance \cite{Alves, ADOM1, AF2, Cao, dFMR, DOS} and references therein.

On the other hand, our work is strongly motivated by some recent existence and multiplicity results established for fractional magnetic problems set in $\R^{N}$ or in bounded domains, involving $(-\Delta)^{s}_{A}$, with $s\in (0, 1)$ and in higher dimension $N>2s$; see \cite{A3, AD, FPV, ZSZ}.

Therefore, the aim of this paper is to give a first contribute concerning the existence, multiplicity and concentration of solutions for fractional magnetic problems in $\R$ with exponential critical growth. 

In order to state precisely our main result, we introduce the assumptions on the potential $V$ and the nonlinearity $f$. Along the paper, we will assume that 
$V\in C(\R, \R)$ verifies the following conditions due to del Pino and Felmer \cite{DF}:
\begin{compactenum}[$(V_1)$]
\item $\inf_{x\in \R} V(x)=V_{0}>0$;
\item there exists a bounded open set $\Lambda\subset \R$ such that
$$
V_{0}<\min_{x\in \partial \Lambda} V(x) \quad \mbox{ and } \quad M=\{x\in \Lambda: V(x)=V_{0}\}\neq \emptyset,
$$
\end{compactenum}
while the function $f:\R\rightarrow \R$ satisfies the following assumptions:
\begin{compactenum}[$(f_1)$]
\item $f\in C^{1}(\R, \R)$ and $f(t)=0$ for  $t\leq 0$;
\item there exists $\alpha_{0}>0$ such that
\begin{align*}
\lim_{t\rightarrow \infty} \frac{f(t)\sqrt{t}}{e^{\alpha t}}=0 \quad \mbox{ for any } \alpha>\alpha_{0}, \quad   \lim_{t\rightarrow \infty} \frac{f(t)\sqrt{t}}{e^{\alpha t}}=\infty \quad \mbox{ for any } \alpha<\alpha_{0};
\end{align*}
\item there exists $\theta>2$ such that $0<\frac{\theta}{2} F(t)\leq t f(t)$ for any $t>0$, where $\displaystyle{F(t)=\int_{0}^{t} f(\tau)d\tau}$;
\item  $f(t)$ is increasing for $t>0$;
\item there exists $p>2$ and $C_{p}>0$ such that 
$$
f(t)\geq C_{p}t^{\frac{p-2}{2}},
$$
where
\begin{align*}
C_{p}>\left[\beta_{p} \left(\frac{2\theta}{\theta-2}  \right) \frac{1}{\min\{1, V_{0}\}}   \right]^{\frac{p-2}{2}},
\end{align*}
with
\begin{align*}
\beta_{p}=\inf_{\mathcal{M}_{0}} I_{0}, \quad \mathcal{M}_{0}=\{u\in X_{0}\setminus\{0\}: \langle I'_{0}(u),u\rangle=0\}, 
\end{align*}
and
\begin{align*}
I_{0}(u)=\frac{1}{2}\left(\frac{1}{2\pi}[u]^{2}+\int_{\R} V_{0}|u|^{2} dx\right)-\frac{1}{p}\int_{\R} |u|^{p} dx;
\end{align*}
\item there exist $\sigma\in (2, \infty)$ and $C_{\sigma}>0$ such that
$$
f'(t)\geq C_{\sigma} t^{\frac{\sigma-4}{2}} \quad \mbox{ for any } t>0.
$$
\end{compactenum} 

\noindent
The main result of this paper is the following:
\begin{thm}\label{thm1}
Suppose that $V$ satisfies $(V_1)$-$(V_2)$ and $f$ verifies $(f_1)$-$(f_6)$. Then, for any $\delta>0$ such that
$$
M_{\delta}=\{x\in \R: {\rm dist}(x, M)\leq \delta\}\subset \Lambda,
$$
there exists $\e_{\delta}>0$ such that, for any $\e\in (0, \e_{\delta})$, problem \eqref{P} has at least $cat_{M_{\delta}}(M)$ nontrivial solution $u_{\e}$. Moreover, if $u_{\e}$ denotes one of these solutions and $\eta_{\e}\in \R$ is a global maximum point of  $|u_{\e}|$, then
$$
\lim_{\e\rightarrow 0} V(\eta_{\e})=V_{0},
$$
and there exists $C>0$ such that
$$
|u_{\e}(x)|\leq \frac{C\e^{2}}{\e^{2}+|x-\eta_{\e}|^{2}} \quad \mbox{ for any } x\in \R.
$$
\end{thm}
We note that when $A=0$ in \eqref{P}, and $V$ and $f$ fulfill $(V_1)$-$(V_2)$ and $(f_1)$-$(f_4)$ respectively, Alves et al. \cite{ADOM}, inspired by \cite{DOS}, obtained the existence of a positive solution which concentrates around a local minima of $V(x)$ as $\e\ri 0$. 
Therefore, Theorem \ref{thm1} can be seen as a generalization in fractional magnetic setting of the results in \cite{ADOM} and \cite{DOS}. Moreover, our results complement and improve them, because here we also consider the question related to the multiplicity of solutions to \eqref{P} for $\e>0$ small enough, which is not considered in the above mentioned papers.
We emphasize that the presence of the magnetic field creates  substantial difficulties that make the study of \eqref{P} rather tough with respect to \cite{ADOM, DOS}, and some appropriate arguments will be needed to overcome this obstacle. 
More precisely, after considering a modified problem in the spirit of \cite{DF} (see also \cite{AFF}), we will make use of the diamagnetic inequality \cite{DS}, the Moser-Trudinger inequality for the modulus of functions belonging to the fractional magnetic space $H^{1/2}_{\e}$ (this is lawful because the diamagnetic inequality says that $|u|\in H^{1/2}(\R, \R)$ whenever $u\in H^{1/2}_{A}(\R, \C)$), and the H\"older continuity of the magnetic field to deduce some fundamental estimates which allow us to deduce the existence of a nontrivial solution for the modified problem when $\e>0$ is small enough.
After that, we apply Nehari manifold argument and  Ljusternik-Schnirelmann category theory to obtain a multiplicity result for the auxiliary problem. Finally, we have to prove that the solutions of the modified problem are also solutions of the original one provided that $\e$ is sufficiently small. This goal is gained by proving some $L^{\infty}$-estimates, independent of $\e$, obtained by combining a Moser iteration procedure \cite{Moser} which takes care of the exponential critical growth of the nonlinearity, with a Kato's type inequality \cite{Kato} which allow us to show that the modulus of each solution of the modified problem is a subsolution of a certain fractional problem in $\R$ involving $(-\Delta)^{1/2}$. 

\noindent
The paper is structured as follows. In Section $2$ we introduce the notations and we prove some technical results which will be used along the paper. In Section $3$ we introduce the modified problem and we use mountain pass theorem to deduce a first existence result for it. In Section $4$ we apply Ljusternik-Schnirelmann category theory to obtain multiple solutions. Finally, we give the proof of Theorem \ref{thm1}.

\section{Preliminaries results}

\noindent
Let $L^{2}(\R, \C)$ be the space of complex-valued functions with summable square, endowed with the real
scalar product 
$$
\langle u, v\rangle_{L^{2}}=\Re\left(\int_{\R} u \bar{v} \,dx\right)
$$
for all $u, v\in L^{2}(\R, \C)$.
We define the fractional magnetic Sobolev space
$$
H^{1/2}_{A}(\R, \C)=\{u\in L^{2}(\R, \C): [u]_{A}<\infty\},
$$
where
$$
[u]_{A}^{2}=\frac{1}{2\pi}\iint_{\R^{2}} \frac{|u(x)-u(y)e^{\imath A(\frac{x+y}{2})\cdot (x-y)}|^{2}}{|x-y|^{2}} \, dx dy,
$$
endowed with the norm
$$
\|u\|^{2}_{A}=[u]_{A}^{2}+\|u\|^{2}_{L^{2}(\R)}.
$$
When $A=0$ and $u: \R\ri \R$, we put 
$$
[u]^{2}:=[u]^{2}_{0}=\frac{1}{2\pi}\iint_{\R^{2}} \frac{|u(x)-u(y)|^{2}}{|x-y|^{2}} \, dx dy
$$ 
which is the Gagliardo seminorm of $u\in H^{1/2}(\R, \R):=\{u\in L^{2}(\R, \R): [u]<\infty\}$. 
Here we used the notation $\|\cdot\|_{L^{p}(\R)}$ to denote the $L^{p}(\R)$-norm of any function $u:\R\rightarrow \R$.
We recall that $H^{1/2}(\R, \R)$ is continuously embedded in $L^{q}(\R, \R)$ for all $q\in [2, \infty)$ and that $H^{1/2}(K, \R)\subset L^{q}(K, \R)$ is compact for any $q\in [1, \infty)$ and compact set $K\subset \R$; see \cite{DPV, MBRS}.

We note that $H^{1/2}_{A}(\R, \C)$ is a Hilbert space with the real scalar product
$$
\langle u, v\rangle_{s, A}=\langle u, v\rangle_{L^{2}}+\frac{1}{2\pi}\Re\iint_{\R^{2}} \frac{(u(x)-u(y)e^{\imath A(\frac{x+y}{2})\cdot (x-y)})\overline{(v(x)-v(y)e^{\imath A(\frac{x+y}{2})\cdot (x-y)})}}{|x-y|^{2}} \,dx dy
$$ 
for any $u, v\in H^{1/2}_{A}(\R, \C)$.

Arguing as in Lemma $3.1$ and  Lemma $3.5$ in \cite{DS}, and Lemma $2.4$ in \cite{AD}, we can see that the following results hold true.
\begin{thm}\label{Sembedding}
The space $H^{1/2}_{A}(\R, \C)$ is continuously embedded into $L^{r}(\R, \C)$ for any $r\in [2, \infty)$ and compactly embedded into $L^{r}(K, \C)$ for any $r\in [1, \infty)$ and any compact $K\subset \R$.
\end{thm}

\begin{lem}\label{DI}
For any $u\in H^{1/2}_{A}(\R, \C)$, we get $|u|\in H^{1/2}(\R,\R)$ and it holds
$$
[|u|]\leq [u]_{A}.
$$
We also have the following pointwise diamagnetic inequality 
$$
\left| |u(x)|-|u(y)| \right|\leq \left|u(x)-u(y)e^{\imath A(\frac{x+y}{2})\cdot (x-y)} \right| \quad \mbox{ a.e. } x, y\in \R.
$$
\end{lem}

\begin{lem}\label{aux}
If $u\in H^{1/2}(\R, \R)$ and $u$ has compact support, then $w=e^{\imath A(0)\cdot x} u \in H^{1/2}_{A}(\R, \C)$.
\end{lem}

We recall the following version of vanishing Lemma of Lions whose proof can be found in Lemma 2.4 in \cite{dSA}:
\begin{lem}\label{lions lemma}
If $(u_{n})$ is a bounded sequence in $H^{1/2}(\R, \R)$ and if
$$
\lim_{n \rightarrow \infty} \sup_{y\in \R} \int_{y-R}^{y+R} |u_{n}|^{2} dx=0
$$
for some $R>0$, then $u_{n}\rightarrow 0$ in $L^{t}(\R^{N})$ for all $t\in (2, \infty)$.
\end{lem}

Next, we prove a list of technical results which will be used along the paper.
\begin{lem}\label{lem2.1ADOM}
Let $(u_{n})\subset H^{1/2}_{A}(\R, \C)$ be a bounded sequence, and set $\sup_{n\in \mathbb{N}} \|u_{n}\|_{\e} = M$. Then, 
\begin{equation*}
\sup_{n\in \mathbb{N}} \int_{\R} (e^{\alpha |u_{n}|^{2}}-1) \, dx <\infty, \quad \mbox{ for every } \alpha \in \left(0, \frac{\omega}{M^{2}}\right). 
\end{equation*}
In particular, if $M\in (0, 1)$, there exists $\alpha_{M}>\omega$ such that  
\begin{equation*}
\sup_{n\in \mathbb{N}} \int_{\R} (e^{\alpha_{M} |u_{n}|^{2}}-1) \, dx <\infty.
\end{equation*}
\end{lem}

\begin{proof}
Fix $\alpha \in \left(0, \frac{\omega}{M^{2}}\right)$. Using Lemma \ref{DI} we can see that
$$
\left \|(-\Delta)^{1/4} \frac{|u_{n}|}{\|u_{n}\|_{A}} \right\|_{L^{2}(\R)}^{2} = \frac{\|(-\Delta)^{1/4} |u_{n}|\|_{L^{2}(\R)}^{2}}{\|u_{n}\|_{A}^{2}}= \frac{[|u_{n}|]^{2}}{\|u_{n}\|_{A}^{2}}\leq 1. 
$$
Then, applying Lemma \ref{MT} we deduce that 
\begin{align}\label{2.9DOMS}
\int_{\R} (e^{\alpha |u_{n}|^{2}}-1) \,dx \leq \int_{\R} \left(e^{\alpha M^{2} \left( \frac{|u_{n}|}{\|u_{n}\|_{\e}}\right)^{2}}-1 \right) \,dx\leq C_{\alpha_{M^{2}}} \frac{\|u_{n}\|_{L^{2}(\R)}^{2}}{\|u_{n}\|_{A}^{2}} \leq C_{\alpha_{M^{2}}}. 
\end{align}
If $M\in (0, 1)$, then $\omega < \frac{\omega}{M^{2}}$, so we can find $\alpha_{M}\in \left( 0, \frac{\omega}{M^{2}}\right)$. Hence, we can use the above inequality replacing $\alpha$ by $\alpha_{M}$.  
\end{proof}

\begin{lem}\label{lem2.2ADOM}
Let $(u_{n})\subset H^{1/2}_{A}(\R, \C)$ be a sequence with
\begin{align}\label{2.2ADOM}
\limsup_{n\ri \infty} \|u_{n}\|_{A}^{2}<1. 
\end{align}
Then, there exists $t>1$ sufficiently close to $1$ and $C>0$ such that 
\begin{align*}
\int_{\R} (e^{\omega |u_{n}|^{2}}- 1)^{t}\, dx \leq C \quad \mbox{ for all } n\in \mathbb{N}. 
\end{align*}
\end{lem}

\begin{proof}
In view of \eqref{2.2ADOM} we can find $m\in (0, 1)$ and $n_{0}\in \mathbb{N}$ such that 
\begin{align*}
\|u_{n}\|_{A}^{2}<m<1 \quad \mbox{ for all } n\geq n_{0}. 
\end{align*}
Take $t>1$ sufficiently close to $1$ and $\beta>t$ such that $\beta m<1$. Recalling that 
\begin{align*}
\left(e^{\omega s^{2}} -1 \right)^{t} \leq C (e^{\omega \beta s^{2}}-1) \quad \mbox{ for all } s\in \R, 
\end{align*}
for some $C= C_{\beta}>0$, we can deduce that 
\begin{align*}
\int_{\R} \left(e^{\omega |u_{n}|^{2}}- 1 \right)^{t}\, dx\leq C \int_{\R} \left(e^{\beta m \omega \left(\frac{|u_{n}|}{\|u_{n}\|_{A}}\right)^{2}}- 1 \right)\, dx \quad \mbox{ for all } n\geq n_{0}.
\end{align*}
Then, applying Lemma \ref{lem2.1ADOM} we can see that 
\begin{align*}
\int_{\R} \left(e^{\omega |u_{n}|^{2}}- 1 \right)^{t}\, dx\leq C_{0} \quad \mbox{ for all } n\geq n_{0},
\end{align*}
for some $C_{0}>0$. Taking 
$$
C= \max \left\{C_{0}, \int_{\R} \left(e^{\omega |u_{1}|^{2}}- 1 \right)^{t}\, dx, \dots, \int_{\R} \left(e^{\omega |u_{n_{0}}|^{2}}- 1\right)^{t}\, dx \right \}
$$
we obtain the thesis. 
\end{proof}

\begin{lem}\label{cor2.3ADOM}
Let $(u_{n})\subset H^{1/2}_{A}(\R, \C)$ be a sequence satisfying \eqref{2.2ADOM}. If $u_{n} \rightharpoonup u$ in $H^{1/2}_{A}(\R, \C)$, then, up to a subsequence, we have 
\begin{align}
&\int_{-R}^{R} F(|u_{n}|^{2})\, dx \ri  \int_{-R}^{R} F(|u|^{2}) \, dx, \label{2.3ADOM} \\
&\int_{-R}^{R} f(|u_{n}|^{2})|u_{n}|^{2} \ri \int_{-R}^{R} f(|u|^{2})|u|^{2}\, dx, \label{2.4ADOM} \\
&\Re \int_{-R}^{R} f(|u_{n}|^{2}) u_{n}\overline{\phi} \, dx \ri \Re \int_{-R}^{R} f(|u|^{2}) u\overline{\phi} \, dx \quad \mbox{ for all } \phi \in \h \mbox{ and } R>0. \label{2.5ADOM}
\end{align}
\end{lem}

\begin{proof}
Using $(f_{1})$ and $(f_{2})$ we can see that for every $\beta>1$ and $\alpha>\alpha_{0}$ there exists $C>0$ such that 
\begin{align*}
|F(t^{2})|\leq C \left(|t|^{2} + (e^{\alpha \beta |t|^{2}} -1) \right) 
\end{align*}
which implies that 
\begin{align}\label{2.6ADOM}
|F(|u_{n}|^{2})|\leq C\left(|u_{n}|^{2} + (e^{\alpha \beta |u_{n}|^{2}} -1)\right).  
\end{align}
Put
\begin{align*}
h_{n}(x):= C \left(e^{\alpha \beta |u_{n}|^{2}} -1\right)
\end{align*}
and fix $\beta, q>1$ sufficiently close to $1$ and $\alpha$ sufficiently close to $\alpha_{0}$ such that $h_{n}\in L^{q}(\R, \R)$ and $\sup_{n\in \mathbb{N}} \|h_{n}\|_{L^{q}(\R)}<\infty$ in view of Lemma \ref{lem2.2ADOM}. Up to a subsequence, we can deduce that 
$$
h_{n}\rightharpoonup h=C(e^{\alpha \beta |u|^{2}} -1)\quad \mbox{ in } L^{q}(\R, \R). 
$$ 
Now, we show that
\begin{align}\label{2.61ADOM}
h_{n}\ri h \quad \mbox{ in } L^{1}(-R, R) \mbox{ for any } R>0. 
\end{align}
Indeed, denoting by $\chi_{R}$ the characteristic function in $(-R, R)$, and observing that $\chi_{R}\in L^{q'}(\R, \R)$, we can see that 
$$
\int_{\R} h_{n}\chi_{R} \,dx\rightarrow \int_{\R} h\chi_{R} \,dx,
$$
that is 
$$
\int_{-R}^{R} h_{n} \,dx\rightarrow \int_{-R}^{R} h \,dx.
$$
Since $h_{n}, h\geq 0$, this means that $\|h_{n}\|_{L^{1}(-R, R)}\rightarrow \|h\|_{L^{1}(-R, R)}$. Observing that $h_{n}\rightarrow h$ a.e. in $\R$, we can use Brezis-Lieb  Lemma to conclude that $\|h_{n}-h\|_{L^{1}(-R, R)}\rightarrow 0$.
Then, recalling that $|u_{n}|\rightarrow |u|$ in $L^{2}(-R, R)$, and using \eqref{2.6ADOM} and \eqref{2.61ADOM}, we can apply the Dominated Convergence Theorem to deduce that \eqref{2.3ADOM} is verified.
Similar arguments show that \eqref{2.4ADOM} and \eqref{2.5ADOM} hold. 
\end{proof}

\begin{lem}\label{vanishingF}
Let $(u_{n})\subset H^{1/2}_{A}(\R, \C)$ be a sequence satisfying \eqref{2.2ADOM}. If there exists $R>0$ such that 
$$
\lim_{n \rightarrow \infty} \sup_{y\in \R} \int_{y-R}^{y+R} |u_{n}|^{2} dx=0, 
$$
then 
$$
\lim_{n\ri \infty} \int_{\R} F(|u_{n}|^{2})\, dx = \lim_{n\ri \infty} \int_{\R} f(|u_{n}|^{2}) |u_{n}|^{2} \, dx =0.
$$
\end{lem}
\begin{proof}
In view of Lemma \ref{lions lemma} we know that 
\begin{align}\label{2.44ADOM}
|u_{n}|\rightarrow 0  \quad\mbox{ in } L^{q}(\R, \R) \mbox{ for any } q\in (2, \infty).
\end{align}
Since $(u_{n})$ verifies \eqref{2.2ADOM}, we can use Lemma \ref{lem2.2ADOM}  to find  $t>1$ sufficiently close to $1$ and $C>0$ such that 
\begin{align*}
\int_{\R} \left(e^{\omega |u_{n}|^{2}}- 1 \right)^{t}\, dx \leq C \quad \mbox{ for all } n\in \mathbb{N}. 
\end{align*}
Then, from the growth assumptions on $f$ and applying H\"older inequality we have
\begin{align*}
\int_{\R} f(|u_{n}|^{2})|u_{n}|^{2}dx&\leq \delta \int_{\R} |u_{n}|^{2}\, dx+C_{\delta} \int_{\R} |u_{n}| (e^{\omega |u_{n}|^{2}}-1) \,dx \\
&\leq C\delta+C\|u_{n}\|_{L^{t'}(\R)} \quad \mbox{ for any } \delta>0, 
\end{align*}
which together with \eqref{2.44ADOM} implies that
$$
\int_{\R} f(|u_{n}|^{2})|u_{n}|^{2}dx\rightarrow 0 \quad \mbox{ as } n\rightarrow \infty.
$$
Consequently, using $(f_3)$ we can deduce that
$$
\int_{\R} F(|u_{n}|^{2}) \,dx\rightarrow 0 \quad \mbox{ as } n\rightarrow \infty.
$$
\end{proof}

\noindent
Finally, we give a variant of Lemma $5$ in \cite{PP} in the one dimensional case.
\begin{lem}\label{lemPSI}
Let $u\in H^{1/2}(\R, \R)$ and $\phi\in \mathcal{C}^{\infty}_{c}(\R, \R)$ such that $0\leq \phi\leq 1$, $\phi=1$ in $(-1,1)$ and $\phi=0$ in $\R\setminus (-2,2)$. Set $\phi_{r}(x)=\phi(\frac{x}{r})$.  Then
$$
\lim_{r\rightarrow \infty} [u \phi_{r}-u]=0 \quad \mbox{ and } \quad \lim_{r\rightarrow \infty} \|u\phi_{r}-u\|_{L^{2}(\R)}=0.
$$
\end{lem}
\begin{proof}
Since $\phi_{r}u\rightarrow u$ a.e. in $\R$ as $r\rightarrow \infty$, $0\leq \phi\leq 1$ and $u\in L^{2}(\R, \R)$, we can use the Dominated Convergence Theorem to see that $\lim_{r\rightarrow \infty} \|u\phi_{r}-u\|_{L^{2}(\R)}=0$.

Therefore, we only need to show the first relation of limit.
Let us note that
{\small\begin{align*}
[u\phi_{r}-u]^{2}&\leq 2 \left(\iint_{\R^{2}}  |u(x)|^{2}\frac{|\phi_{r}(x)-\phi_{r}(y)|^{2}}{|x-y|^{2}}dx dy+ \iint_{\R^{2}} \frac{|\phi_{r}(x)-1|^{2}|u(x)-u(y)|^{2}}{|x-y|^{2}}dx dy\right) \\
&=:2(A_{r}+B_{r}).
\end{align*}}Taking into account that $|\phi_{r}(x)-1|\leq 2$, $|\phi_{r}(x)-1|\rightarrow 0$ a.e. in $\R$ and $u\in H^{\frac{1}{2}}(\R, \R)$, and applying the Dominated Convergence Theorem we get
$$
B_{r}\rightarrow 0 \quad \mbox{ as } r\rightarrow \infty.
$$
Now we show that
$$
A_{r}\rightarrow 0 \quad \mbox{ as } r\rightarrow \infty.
$$
Firstly, we observe that $\R^{2}$ can be written as
{\small\begin{align*}
\R^{2}&=((\R \setminus (-2r, 2r))\times (\R \setminus (-2r, 2r)))\cup ((-2r, 2r)\times \R) \cup ((\R \setminus (-2r, 2r))\times (-2r, 2r))\\
&=: X^{1}_{r}\cup X^{2}_{r} \cup X^{3}_{r},
\end{align*}}
so that
\begin{align}\label{Pa1SA}
&\iint_{\R^{2}} |u(x)|^{2} \frac{|\phi_{r}(x)-\phi_{r}(y)|^{2}}{|x-y|^{2}} \, dx dy \nonumber\\
&=\iint_{X^{1}_{r}} |u(x)|^{2} \frac{|\phi_{r}(x)-\phi_{r}(y)|^{2}}{|x-y|^{2}} \, dx dy
+\iint_{X^{2}_{r}} |u(x)|^{2} \frac{|\phi_{r}(x)-\phi_{r}(y)|^{2}}{|x-y|^{2}} \, dx dy \nonumber\\
&+ \iint_{X^{3}_{r}} |u(x)|^{2} \frac{|\phi_{r}(x)-\phi_{r}(y)|^{2}}{|x-y|^{2}} \, dx dy.
\end{align}
Next, we estimate each integral in (\ref{Pa1SA}).
Recalling that $\phi=0$ in $\R\setminus (-2, 2)$, we have
\begin{align}\label{Pa2SA}
\iint_{X^{1}_{r}} |u(x)|^{2} \frac{|\phi_{r}(x)-\phi_{r}(y)|^{2}}{|x-y|^{2}} \, dx dy=0.
\end{align}
Using $0\leq \phi\leq 1$, $\|\phi'\|_{L^{\infty}(\R)}\leq C$ and the Mean Value Theorem, we can see that
\begin{align}\label{Pa3SA}
 &\iint_{X^{2}_{r}} |u(x)|^{2} \frac{|\phi_{r}(x)-\phi_{r}(y)|^{2}}{|x-y|^{2}} \, dx dy \nonumber\\
&=\int_{-2r}^{2r} \,dx \int_{\{y\in \R: |y-x|\leq r\}} |u(x)|^{2} \frac{|\phi_{r}(x)-\phi_{r}(y)|^{2}}{|x-y|^{2}} \, dy \nonumber
&+\int_{-2r}^{2r} \, dx \int_{\{y\in \R: |y-x|> r\}} |u(x)|^{2} \frac{|\phi_{r}(x)-\phi_{r}(y)|^{p}}{|x-y|^{2}} \, dy  \nonumber\\
&\leq C r^{-2} \|\phi'\|_{L^{\infty}(\R)}^{2} \int_{-2r}^{2r} \, dx \int_{\{y\in \R: |y-x|\leq r\}} |u(x)|^{2} \, dy \nonumber \\
&+ C \int_{-2r}^{2r} \, dx \int_{\{y\in \R: |y-x|> r\}} \frac{|u(x)|^{2}}{|x-y|^{2}} \, dy \nonumber\\
&\leq C r^{-1} \int_{-2r}^{2r} |u(x)|^{2} \, dx+C r^{-1} \int_{-2r}^{2r} |u(x)|^{2} \, dx \nonumber \\
&=Cr^{-1} \int_{-2r}^{2r} |u(x)|^{2} \, dx.
\end{align}
Concerning the integral on $X^{3}_{r}$, we can note that
\begin{align}\label{Pa4SA}
&\iint_{X^{3}_{r}} |u(x)|^{2} \frac{|\phi_{r}(x)-\phi_{r}(y)|^{2}}{|x-y|^{2}} \, dx dy \nonumber\\
&=\int_{\R\setminus (-2r, 2r)} \, dx \int_{\{y\in (-2r, 2r): |y-x|\leq r\}} |u(x)|^{2} \frac{|\phi_{r}(x)-\phi_{r}(y)|^{2}}{|x-y|^{2}} \, dy \nonumber\\
&+\int_{\R\setminus (-2r, 2r)} \,dx \int_{\{y\in (-2r, 2r): |y-x|>r\}} |u(x)|^{2} \frac{|\phi_{r}(x)-\phi_{r}(y)|^{2}}{|x-y|^{2}} \, dy=: C_{r}+ D_{r}.
\end{align}
Applying the Mean Value Theorem and observing that if $(x, y) \in (\R\setminus (-2r, 2r))\times (-2r, 2r)$ and $|x-y|\leq r$ then $|x|\leq 3r$, we get
\begin{align}\label{Pa5SA}
C_{r}&\leq r^{-2} \|\phi'\|_{L^{\infty}(\R)}^{2} \int_{-3r}^{3r} \, dx \int_{\{y\in (-2r, 2r): |y-x|\leq r\}} |u(x)|^{2} \, dy \nonumber\\
&\leq C r^{-1} \int_{-3r}^{3r} |u(x)|^{2} \, dx.
\end{align}
Now, we can see that for any $K>4$ it holds
$$
X_{r}^{3}=(\R\setminus (-2r, 2r))\times (-2r, 2r) \subset ( (-Kr, Kr) \times (-2r, 2r)) \cup ((\R\setminus (-Kr, Kr))\times (-2r, 2r)).
$$
Therefore
\begin{align}\label{Pa6SA}
&\int_{-Kr}^{Kr} \, dx \int_{\{y\in (-2r, 2r): |y-x|> r\}}  |u(x)|^{2} \frac{|\phi_{r}(x)-\phi_{r}(y)|^{2}}{|x-y|^{2}} \, dy \nonumber\\
&\quad \leq C \int_{-Kr}^{Kr} \, dx \int_{\{y\in (-2r, 2r): |y-x|> r\}} \frac{|u(x)|^{2}}{|x-y|^{2}} \, dy \nonumber \\
&\quad \leq C \int_{-Kr}^{Kr} |u(x)|^{2} \, dx \int_{\{z\in \R: |z|> r\}} \frac{1}{|z|^{2}} \, dz \nonumber\\
&\quad = C r^{-1} \int_{-Kr}^{Kr} |u(x)|^{2} \, dx.
\end{align}
Let us note that if $(x, y)\in (\R\setminus (-Kr, Kr))\times (-r, r)$, then 
$$
|x-y|\geq |x|- |y|\geq \frac{|x|}{2}+ \frac{K}{2}r -2r >\frac{|x|}{2}.
$$ 
Fix $\gamma>1$. Applying the H\"older inequality with exponents $\gamma$ and $\gamma'$ we obtain
\begin{align}\label{Pa7SA}
&\int_{\R\setminus (-Kr, Kr)} \, dx \int_{\{y\in (-2r, 2r): |y-x|>r\}} |u(x)|^{2} \frac{|\phi_{r}(x)-\phi_{r}(y)|^{2}}{|x-y|^{2}} \, dy \nonumber\\
&\leq  C \int_{\R\setminus (-Kr, Kr)} \, dx \int_{\{y\in (-2r, 2r): |y-x|>r \}} \frac{|u(x)|^{2}}{|x-y|^{2}} \, dy \nonumber\\
&\leq C r \int_{\R\setminus (-Kr, Kr)} \frac{|u(x)|^{2}}{|x|^{2}} \, dx \nonumber\\
&\leq C r \left(\int_{\R\setminus (-Kr, Kr)} |u(x)|^{2\gamma'} \, dx\right)^{\frac{1}{\gamma'}} \left(\int_{\R\setminus (-Kr, Kr)} |x|^{-2\gamma} \, dx\right)^{\frac{1}{\gamma}} \nonumber\\
&\leq C K^{1-2\gamma} r^{2-2\gamma}\left(\int_{\R\setminus (-Kr, Kr)} |u(x)|^{2\gamma'} \, dx\right)^{\frac{1}{\gamma'}}.
\end{align}
Taking into account (\ref{Pa6SA}), (\ref{Pa7SA}) and $H^{1/2}(\R, \R)\subset L^{2\gamma'}(\R, \R)$, we deduce
\begin{align}\label{Pa8SA}
D_{r}\leq C r^{-1} \int_{-Kr}^{Kr} |u(x)|^{2} \, dx+C K^{1-2\gamma} r^{2-2\gamma}.
\end{align}
Putting together (\ref{Pa1SA})-(\ref{Pa5SA}), (\ref{Pa8SA}) and using $K>4$ and $1-2\gamma<0$ we have
\begin{align*}
\iint_{\R^{2}} |u(x)|^{2} \frac{|\phi_{r}(x)-\phi_{r}(y)|^{2}}{|x-y|^{2}} \, dx dy &\leq Cr^{-1} \int_{-Kr}^{Kr} |u(x)|^{2} \, dx+C K^{1-2\gamma} r^{2-2\gamma}\\
&\leq Cr^{-1}+CK^{1-2\gamma} r^{2-2\gamma}\\
&\leq Cr^{-1}+C4^{1-2\gamma} r^{2-2\gamma}.
\end{align*}
Since $2-2\gamma<0$, we deduce
\begin{align*}
& \limsup_{r\rightarrow \infty} \iint_{\R^{2}} |u(x)|^{2} \frac{|\phi_{r}(x)-\phi_{r}(y)|^{2}}{|x-y|^{2}} \, dx dy =0.
\end{align*}
This ends the proof of lemma.
\end{proof}

\section{Variational framework and modified problem}

\noindent
Using the change of variable $u(x)\mapsto u(\e x)$, instead of (\ref{P}), we deal with the following equivalent problem
\begin{equation}\label{Pe}
(-\Delta)^{1/2}_{A_{\e}} u + V_{\e}(x)u =  f(|u|^{2})u \quad \mbox{ in } \R,
\end{equation}
where $A_{\e}(x)=A(\e x)$ and $V_{\e}(x)=V(\e x)$. \\
Hereafter, we use the penalization argument introduced in \cite{DF} (see also \cite{AFF, DOS}) to study \eqref{Pe}.

\noindent
Fix $k>\frac{2\theta}{\theta-2}$ and $a>0$ such that $f(a)=\frac{V_{0}}{k}$, and we consider the function
$$
\hat{f}(t):=
\begin{cases}
f(t)& \text{ if $t \leq a$} \\
\frac{V_{0}}{k}    & \text{ if $t >a$}.
\end{cases}
$$ 
Let $t_{a}, T_{a}>0$ be such that $t_{a}<a<T_{a}$ and take $\xi\in C^{\infty}_{c}(\R, \R)$ such that
\begin{compactenum}[$(\xi_1)$]
\item $\xi(t)\leq \hat{f}(t)$ for all $t\in [t_{a}, T_{a}]$,
\item $\xi(t_{a})=\hat{f}(t_{a})$, $\xi(T_{a})=\hat{f}(T_{a})$, $\xi'(t_{a})=\hat{f}'(t_{a})$ and $\xi'(T_{a})=\hat{f}'(T_{a})$, 
\item the map $t\mapsto \xi(t)$ is increasing for all $t\in [t_{a}, T_{a}]$.
\end{compactenum}
Let us define $\tilde{f}\in C^{1}(\R, \R)$ as follows:
$$
\tilde{f}(t):=
\begin{cases}
\hat{f}(t)& \text{ if $t\notin [t_{a}, T_{a}]$} \\
\xi(t)    & \text{ if $t\in [t_{a}, T_{a}]$}.
\end{cases}
$$ 
We introduce the following penalized nonlinearity $g: \R\times \R\rightarrow \R$ by 
$$
g(x, t)=\chi_{\Lambda}(x)f(t)+(1-\chi_{\Lambda}(x))\tilde{f}(t),
$$
where $\chi_{\Lambda}$ is the characteristic function on $\Lambda$, and  we write $\displaystyle{G(x, t)=\int_{0}^{t} g(x, \tau)\, d\tau}$.

In the light of $(f_1)$-$(f_4)$, it follows that $g$ verifies the following properties:
\begin{compactenum}[($g_1$)]
\item $\displaystyle{\lim_{t\rightarrow 0} g(x, t)=0}$ uniformly in $x\in \R$;
\item fixed $q\geq 2$, $\delta>0$ and $\beta>1$ there exists $C_{\delta}>0$ such that
\begin{align*}
g(x, t)\leq \delta+C_{\delta} t^{\frac{q-2}{2}} \left(e^{\beta \omega t}-1 \right), \quad \mbox{ for any } t\geq 0,
\end{align*}
and
\begin{align*}
G(x, t)\leq \delta t+C_{\delta} t^{\frac{q}{2}}(e^{\beta \omega t}-1), \quad \mbox{ for any } t\geq 0,
\end{align*}
\item $(i)$ $0< \frac{\theta}{2} G(x, t)\leq g(x, t)t$ for any $x\in \Lambda$ and $t>0$, \\
$(ii)$ $0\leq  G(x, t)\leq g(x, t)t\leq \frac{V(x)}{k}t$ and $0\leq g(x,t)\leq \frac{V(x)}{k}$ for any $x\in \Lambda^{c}$ and $t>0$;
\item $t\mapsto g(x,t)$ is increasing for $t>0$.
\end{compactenum}

From now on, we focus our study on the following modified problem 
\begin{equation}\label{MPe}
(-\Delta)^{1/2}_{A_{\e}} u + V_{\e}(x)u =  g_{\e}(x, |u|^{2})u \quad \mbox{ in } \R, 
\end{equation}
where $g_{\e}(x, t)=g(\e x, t)$. 
Let us note that if $u$ is a solution of (\ref{MPe}) such that 
\begin{equation}\label{ue}
|u(x)|\leq t_{a} \quad \mbox{ for all } x\in  \R\setminus\Lambda_{\e},
\end{equation}
where $\Lambda_{\e}:=\{x\in \R: \e x\in \Lambda\}$, then $u$ is also a solution of the original problem  (\ref{Pe}).

In order to study weak solutions to (\ref{MPe}), we seek critical points of the Euler-Lagrange functional
$$
J_{\e}(u)=\frac{1}{2}\|u\|^{2}_{\e}-\frac{1}{2}\int_{\R} G_{\e}(x, |u|^{2})\, dx
$$
which is well-defined for any function $u$ belonging to the space
$$
\h=\overline{C^{\infty}_{c}(\R, \C)}^{\|\cdot\|_{\e}}
$$
endowed with the norm 
$$
\|u\|^{2}_{\e}=\frac{1}{2\pi}[u]^{2}_{A_{\e}}+\|\sqrt{V_{\e}} \, |u|\|^{2}_{L^{2}(\R)}.
$$
Using $(V_1)$, it is easy to see that the embedding $\h\subset L^{q}(\R, \C)$ is continuous for all $q\in [2, \infty)$ and  $\h\subset L^{q}(K, \C)$ for all $q\in [1, \infty)$ and $K\subset \R$ compact.

As we will see later, it will be fundamental to consider the following family of autonomous problem associated with \eqref{Pe}, that is, for $\mu\in \R_{+}$
\begin{equation}\label{AP0}
(-\Delta)^{1/2} u + \mu u =  f(u^{2})u \quad \mbox{ in } \R.
\end{equation}
We denote by $J_{\mu}: H^{1/2}_{\mu}(\R, \R)\rightarrow \R$ the corresponding energy functional
\begin{align*}
J_{\mu}(u)=\frac{1}{2}\|u\|^{2}_{\mu}-\frac{1}{2}\int_{\R} F(u^{2})\, dx, 
\end{align*}
where $H^{1/2}_{\mu}(\R, \R)$ is the space $H^{1/2}(\R, \R)$ equipped with the norm 
$$
\|u\|^{2}_{\mu}=\frac{1}{2\pi}[u]^{2}+\mu \|u\|^{2}_{L^{2}(\R)}.
$$

Firstly, we can note that $J_{\e}$ possesses a mountain pass structure \cite{AR}.
\begin{lem}\label{MPG}
\begin{compactenum}[$(i)$]
\item $J_{\e}(0)=0$;
\item there exist $\sigma, \rho>0$ such that $J_{\e}(u)\geq \sigma$ for any $u\in \h$ such that $\|u\|_{\e}=\rho$;
\item there exists $e\in \h$ with $\|e\|_{\e}>\rho$ such that $J_{\e}(e)<0$.
\end{compactenum}
\end{lem}
\begin{proof}
Fix $u\in \h$ such that $\|u\|_{\e}= \rho<1$ and take $\alpha \in (\omega, \frac{\omega}{\rho^{2}})$. From the growth assumptions on $g$, there exists $r>1$ close to $1$ such that $r\alpha< \frac{\omega}{\rho^{2}}$, $q>2$ and $C>0$ with 
\begin{align*}
G_{\e}(x, t^{2}) \leq \frac{V_{0}}{4} t^{2} + C \left(e^{r\alpha t^{2}} -1\right)^{\frac{1}{r}} t^{q} \quad \mbox{ for all } t\geq 0.  	
\end{align*}
Therefore, applying H\"older inequality and using \eqref{2.9DOMS} we get
\begin{align*}
J_{\e}(u)&\geq \frac{1}{2}\|u\|^{2}_{\e}-\frac{V_{0}}{4} \|u\|^{2}_{L^{2}(\R)}-C \int_{\R} \left(e^{r\alpha |u|^{2}} -1 \right)^{\frac{1}{r}} |u|^{q} dx \\
&\geq \frac{1}{4}\|u\|^{2}_{\e}- C \left( \int_{\R} (e^{r\alpha |u|^{2}} -1)\, dx \right)^{\frac{1}{r} }\left( \int_{\R} |u|^{r' q} \, dx \right)^{\frac{1}{r'}}\\
&\geq \frac{1}{4}\|u\|^{2}_{\e}- C \|u\|_{\e}^{q}= \frac{1}{4}\rho^{2} - C\rho^{q}= \sigma>0, 
\end{align*}
for every $\rho$ sufficiently small. 

Regarding $(iii)$, we can note that in view of $(g_3)$, we have for any $u\in \h\setminus\{0\}$ with $\supp(u)\subset \Lambda_{\e}$ and $t>0$
\begin{align*}
J_{\e}(tu)&\leq \frac{t^{2}}{2} \|u\|^{2}_{\e}-\frac{1}{2}\int_{\Lambda_{\e}} G_{\e}(x, t^{2}|u|^{2})\, dx \\
&\leq \frac{t^{2}}{2} \|u\|^{2}_{\e}-Ct^{\theta} \int_{\Lambda_{\e}} |u|^{\theta}\, dx+C
\end{align*}
which together with $\theta>2$ implies that $J_{\e}(tu)\rightarrow -\infty$ as $t\rightarrow \infty$.
\end{proof}

Taking into account Lemma \ref{MPG}, we can define the mountain pass level
$$
c_{\e}=\inf_{\gamma\in \Gamma_{\e}} \max_{t\in [0, 1]} J_{\e}(\gamma(t)), 
$$
where
$$
\Gamma_{\e}=\{\gamma\in C([0, 1], \h): \gamma(0)=0 \mbox{ and } J_{\e}(\gamma(1))<0\}.
$$
Let us introduce the Nehari manifold associated with (\ref{MPe}), that is
\begin{equation*}
\mathcal{N}_{\e}:= \{u\in \h \setminus \{0\} : \langle J_{\e}'(u), u \rangle =0\}.
\end{equation*}
We note that for all $u\in \N_{\e}$, from the growth assumptions on $g$, it follows that we can find $r^{*}>0$ (independent of $u$) such that
\begin{align}\label{uNr}
\|u\|_{\e}\geq r^{*}>0.
\end{align}
We denote by $c_{\mu}$ and $\mathcal{N}_{\mu}$ the mountain pass level and the Nehari manifold associated with \eqref{AP0}, respectively.
It is easy to verify (see \cite{W}) that $c_{\e}$ can be characterized as follows:
$$
c_{\e}=\inf_{u\in \h\setminus\{0\}} \sup_{t\geq 0} J_{\e}(t u)=\inf_{u\in \N_{\e}} J_{\e}(u).
$$

\noindent
In view of Theorem $1.1$ in \cite{dSA}, we can state the following compactness result for the autonomous problem \eqref{AP0}:
\begin{lem}\label{FS}
Let $(u_{n})\subset \mathcal{N}_{\mu}$ be a sequence satisfying $J_{\mu}(u_{n})\rightarrow c_{\mu}$. Then we have either
\begin{compactenum}[(i)]
\item $(u_{n})$ strongly converges in $H^{1/2}_{\mu}(\R, \R)$, or
\item there exists a sequence $(\tilde{y}_{n})\subset \R$ such that,  up to a subsequence, $v_{n}(x)=u_{n}(x+\tilde{y}_{n})$ converges strongly in $H^{1/2}_{\mu}(\R, \R)$.
\end{compactenum}
In particular, there exists a nonnegative minimizer $w$ of $J_{\mu}$ in $\mathcal{N}_{\mu}$.
\end{lem}
\begin{remark}\label{Rodica}
From $(f_5)$ and Lemma $3.2$ in \cite{ADOM}, it follows that $c_{\mu}$ verifies 
$$
c_{\mu}<\min\{1, \mu\} \left(\frac{1}{2}-\frac{1}{\theta} \right).
$$
\end{remark}

\noindent
Next, we show an interesting relation between $c_{\e}$ and $c_{V_{0}}$.
\begin{lem}\label{AMlem1}
The numbers $c_{\e}$ and $c_{V_{0}}$ satisfy the following inequality
$$
\limsup_{\e\rightarrow 0} c_{\e}\leq c_{V_{0}}<\min\{1, V_{0}\} \left(\frac{1}{2}-\frac{1}{\theta} \right).
$$
\end{lem}
\begin{proof}
Arguing as in Theorem $1.1$ in \cite{dSA}, we can find a positive ground state $w\in H^{1/2}_{V_{0}}(\R, \R)$ to autonomous problem \eqref{AP0} with $\mu=V_{0}$, so that $J'_{V_{0}}(w)=0$ and $J_{V_{0}}(w)=c_{V_{0}}$. 

Using a suitable Moser iteration argument (see proof of Lemma \ref{moser}), we can prove that $w\in L^{\infty}(\R, \R)$. This implies that $f(|w|^{2})w\in L^{\infty}(\R, \R)$. 
Applying Proposition $2.9$ in \cite{Silvestre}, we can deduce that $w\in C^{0, \alpha}(\R, \R)$.
Since $f\in C^{1}$ and $w\in L^{\infty}(\R, \R)$, we obtain that $f(|w|^{2})w\in C^{0, \alpha}(\R, \R)$, and by Proposition $2.8$ in \cite{Silvestre}, we can deduce that $w\in C^{1, \alpha}(\R, \R)$. This, together with $w\in L^{2}(\R, \R)$ implies that $w(x)\rightarrow 0$ as $|x|\rightarrow \infty$. 
In what follows, we prove the following decay estimate for $w$:
\begin{equation}\label{remdecay}
0<w(x)\leq \frac{C}{|x|^{2}} \quad \mbox{ for all } |x|>1.
\end{equation}
In the light of $w(x)\rightarrow 0$ as $|x|\rightarrow \infty$ and $(f_{1})$, we can see that there exists $R>0$ such that 
\begin{align}\label{caserta1}
(-\Delta)^{1/2} w + \frac{V_{0}}{2} w \leq 0 \quad \mbox{ in } \R \setminus (-R, R). 
\end{align}
Let $H$ be the fundamental solution to
\begin{align}\label{caserta2}
(-\Delta)^{1/2} H + \frac{V_{0}}{2} H =\delta \quad \mbox{ in } \R. 
\end{align}
From Lemma $C.1.$ in \cite{FLS} (see also \cite{BG, FQT}), we know that $H$ is continuous and $H>0$, so we can deduce that $H(x)\geq c$ for all $x\in [-R, R]$, for some $c>0$. 

Let $C_{0}= \|w\|_{L^{\infty}(\R)} c^{-1}$, and set $z(x)= C_{0} H(x)- w(x)$. Note that $z(x)\geq 0$ in $(-R, R)$.  Let us prove that $z(x)\geq 0$ in $\R\setminus (-R, R)$. If by contradiction $z<0$ somewhere in $\R\setminus (-R, R)$, using the fact that $z(x)\ri 0$ as $|x|\ri \infty$ and $z\geq 0$ in $(-R, R)$, we can see that $z$ attains a strict global minimum at some point $x_{0}\in \R\setminus (-R, R)$ with $z(x_{0})<0$. This, together with the singular integral expression for $(-\Delta)^{1/2}$ yield 
\begin{align*}
(-\Delta)^{1/2} z(x_{0}) + \frac{V_{0}}{2} z(x_{0})<0, 
\end{align*}
and this is a contradiction because \eqref{caserta1} and \eqref{caserta2} give $(-\Delta)^{1/2} z(x_{0}) + \frac{V_{0}}{2} z(x_{0})\geq 0$ in $\R\setminus (-R, R)$. Therefore $z\geq 0$ in $\R$. Recalling that (see \cite{BG, FQT, FLS}) $H(x)\leq \frac{C}{x^{2}}$ for $|x|>1$ big enough, we can conclude that \eqref{remdecay} holds true.

Now, take a cut-off function $\eta\in C^{\infty}_{c}(\R, [0,1])$ such that $\eta=1$ in a neighborhood of zero $(-\frac{\delta}{2}, \frac{\delta}{2})$ and $\supp(\eta)\subset (-\delta, \delta)\subset \Lambda$ for some $\delta>0$. 

Let us define $w_{\e}(x):=\eta_{\e}(x)w(x) e^{\imath A(0)\cdot x}$, with $\eta_{\e}(x)=\eta(\e x)$ for $\e>0$, and we observe that $|w_{\e}|=\eta_{\e}w$ and $w_{\e}\in \h$ in view of Lemma \ref{aux}. 

In what follows, we prove that
\begin{equation}\label{limwr}
\lim_{\e\rightarrow 0}\|w_{\e}\|^{2}_{\e}=\|w\|_{V_{0}}^{2}\in(0, \infty).
\end{equation}
Clearly, from the Dominated Convergence Theorem, we have 
\begin{align*}
\int_{\R} V_{\e}(x)|w_{\e}|^{2}dx\rightarrow \int_{\R} V_{0} |w|^{2}dx.
\end{align*}
Next, we show that
\begin{equation}\label{limwr*}
\lim_{\e\rightarrow 0}[w_{\e}]^{2}_{A_{\e}}=[w]^{2}.
\end{equation}
Let us observe that 
\begin{align*}
[w_{\e}]_{A_{\e}}^{2}
&=\iint_{\R^{2}} \frac{|e^{\imath A(0)\cdot x}\eta_{\e}(x)w(x)-e^{\imath A_{\e}(\frac{x+y}{2})\cdot (x-y)}e^{\imath A(0)\cdot y} \eta_{\e}(y)w(y)|^{2}}{|x-y|^{2}} dx dy \nonumber \\
&=[\eta_{\e} w]^{2}
+\iint_{\R^{2}} \frac{\eta_{\e}^2(y)w^2(y) |e^{\imath [A_{\e}(\frac{x+y}{2})-A(0)]\cdot (x-y)}-1|^{2}}{|x-y|^{2}} dx dy\\
&\quad+2\Re \iint_{\R^{2}} \frac{(\eta_{\e}(x)w(x)-\eta_{\e}(y)w(y))\eta_{\e}(y)w(y)(1-e^{-\imath [A_{\e}(\frac{x+y}{2})-A(0)]\cdot (x-y)})}{|x-y|^{2}} dx dy \\
&=: [\eta_{\e} w]^{2}+X_{\e}+2Y_{\e}, 
\end{align*}
and 
$$
|Y_{\e}|\leq [\eta_{\e} w] \sqrt{X_{\e}}. 
$$ 
On the other hand, by Lemma \ref{lemPSI}, it follows that
\begin{equation}\label{PPlem}
[\eta_{\e} w]\rightarrow [w] \quad \mbox{ as } \e\rightarrow 0. 
\end{equation}
Therefore, it is enough to show that $X_{\e}\rightarrow 0$ as $\e\rightarrow 0$ to infer that \eqref{limwr*} holds.

Fix $0<\beta<\alpha/({\frac{1}{2}+\alpha})$, and we note that
\begin{equation}\label{Ye}
\begin{split}
X_{\e}
&\leq \int_{\R} w^{2}(y) dy \int_{|x-y|\geq\e^{-\beta}} \frac{|e^{\imath [A_{\e}(\frac{x+y}{2})-A(0)]\cdot (x-y)}-1|^{2}}{|x-y|^{2}} dx\\
&+\int_{\R} w^{2}(y) dy  \int_{|x-y|<\e^{-\beta}} \frac{|e^{\imath [A_{\e}(\frac{x+y}{2})-A(0)]\cdot (x-y)}-1|^{2}}{|x-y|^{2}} dx \\
&=:X^{1}_{\e}+X^{2}_{\e}.
\end{split}
\end{equation}
Since $|e^{\imath t}-1|^{2}\leq 4$ and $w\in H^{1/2}_{V_{0}}(\R, \R)$, we obtain
\begin{equation}\label{Ye1}
X_{\e}^{1}\leq C \int_{\R^{N}} w^{2}(y) dy \int_{\e^{-\beta}}^\infty \rho^{-2} d\rho\leq C \e^{2\beta s} \rightarrow 0.
\end{equation}
Taking into account $|e^{\imath t}-1|^{2}\leq t^{2}$ for all $t\in \R$, $A\in C^{0,\alpha}(\R, \R)$ for $\alpha\in(0,1]$, and $|x+y|^{2}\leq 2(|x-y|^{2}+4|y|^{2})$, we can see that
\begin{equation}\label{Ye2}
\begin{split}
X^{2}_{\e}&
	\leq \int_{\R} w^{2}(y) dy  \int_{|x-y|<\e^{-\beta}} \left|A_{\e}\Bigl(\frac{x+y}{2}\Bigr)-A(0) \right|^{2}  dx \\
	&\leq C\e^{2\alpha} \int_{\R} w^{2}(y) dy  \int_{|x-y|<\e^{-\beta}} |x+y|^{2\alpha} dx \\
	&\leq C\e^{2\alpha} \left(\int_{\R} w^{2}(y) dy  \int_{|x-y|<\e^{-\beta}} |x-y|^{2\alpha} dx+  \int_{\R} |y|^{2\alpha} w^{2}(y) dy  \int_{|x-y|<\e^{-\beta}} dx\right) \\
	&=: C\e^{2\alpha} (X^{2, 1}_{\e}+X^{2, 2}_{\e}).
	\end{split}
	\end{equation}	
Now, 
	\begin{equation}\label{Ye21}
	X^{2, 1}_{\e}
	= C  \int_{\R} w^{2}(y) dy \int_0^{\e^{-\beta}} \rho^{2\alpha} d\rho
	\leq C\e^{-\beta(1+2\alpha)}.
	\end{equation}
	On the other hand, by \eqref{remdecay}, we infer that
	\begin{equation}\label{Ye22}
	\begin{split}
	 X^{2, 2}_{\e}
	 &\leq C  \int_{\R} |y|^{2\alpha} w^{2}(y) dy \int_0^{\e^{-\beta}} d\rho  \\
	&\leq C \e^{-\beta} \left(\int_{-1}^{1}  w^{2}(y) dy + \int_{\R\setminus (-1, 1)} \frac{1}{|y|^{4-2\alpha}} dy \right)  \leq C \e^{-\beta}.
	\end{split}
	\end{equation}
Putting together \eqref{Ye}, \eqref{Ye1}, \eqref{Ye2}, \eqref{Ye21} and \eqref{Ye22} we can conclude that $X_{\e}\rightarrow 0$, and then \eqref{limwr} holds true.

Let $t_{\e}>0$ be the unique number such that 
\begin{equation*}
J_{\e}(t_{\e} w_{\e})=\max_{t\geq 0} J_{\e}(t w_{\e}).
\end{equation*}
Then $t_{\e}$ verifies 
\begin{equation}\label{AS1}
\|w_{\e}\|_{\e}^{2}=\int_{\R} g_{\e}(x, t_{\e}^{2} |w_{\e}|^{2}) |w_{\e}|^{2}dx=\int_{\R} f(t_{\e}^{2} |w_{\e}|^{2}) |w_{\e}|^{2}dx, 
\end{equation}
where we used $\supp(\eta)\subset \Lambda$ and $g=f$ on $\Lambda$.
Let us prove that $t_{\e}\rightarrow 1$ as $\e\rightarrow 0$. Since $\eta=1$ in $(-\frac{\delta}{2}, \frac{\delta}{2})$ and recalling that $w$ is a continuous positive function, we can see that $(f_4)$ yields
$$
\|w_{\e}\|_{\e}^{2}\geq f(t_{\e}^{2}\alpha^{2}_{0})\int_{-\frac{\delta}{2}}^{\frac{\delta}{2}} |w|^{2}dx, 
$$
where $\alpha_{0}=\min_{[- \frac{\delta}{2}, \frac{\delta}{2}]} w>0$. 

If $t_{\e}\rightarrow \infty$ as $\e\rightarrow 0$ then we can use $(f_3)$ and \eqref{limwr} to deduce that $\|w\|_{V_{0}}^{2}= \infty$ which gives a contradiction.
On the other hand, if $t_{\e}\rightarrow 0$ as $\e\rightarrow 0$ we can use the growth assumptions on $g$ and \eqref{limwr} to infer that $\|w\|_{V_{0}}^{2}= 0$ which is impossible.

In conclusion $t_{\e}\rightarrow t_{0}\in (0, \infty)$ as $\e\rightarrow 0$.
Hence, taking the limit as $\e\rightarrow 0$ in \eqref{AS1} and using \eqref{limwr}, we can see that 
\begin{equation}\label{AS2}
\|w\|_{V_{0}}^{2}=\int_{\R} f(t_{0}^{2} |w|^{2}) |w|^{2} dx.
\end{equation}
By $w\in \mathcal{N}_{V_{0}}$ and $(f_4)$, we deduce that $t_{0}=1$. Applying the Dominated Convergence Theorem, we obtain that $\lim_{\e\rightarrow 0} J_{\e}(t_{\e} w_{\e})=J_{V_{0}}(w)=c_{V_{0}}$.
Since $c_{\e}\leq \max_{t\geq 0} J_{\e}(t w_{\e})=J_{\e}(t_{\e} w_{\e})$, we can conclude  that
$\limsup_{\e\rightarrow 0} c_{\e}\leq c_{V_{0}}$. This together with Remark \ref{Rodica} gives the thesis.
\end{proof}

In the next lemma we show a compactness condition for $J_{\e}$.
\begin{lem}\label{PSc}
Let $c\in \R$ be such that $c<\min\{1, V_{0}\} \left(\frac{1}{2}-\frac{1}{\theta} \right)$. Then $J_{\e}$ satisfies the Palais-Smale condition at the level $c$.
\end{lem}
\begin{proof}
{\bf Claim1} Every $(PS)_{c}$ sequence is bounded in $\h$.

Let $(u_{n})\subset \h$ be a $(PS)_{c}$ sequence. Then, using $(g_3)$ we have
\begin{align*}
c+o_{n}(1)&= J_{\e}(u_{n})-\frac{1}{\theta}\langle J'_{\e}(u_{n}), u_{n}\rangle \\
&\geq \left(\frac{1}{2}-\frac{1}{\theta}\right)[u_{n}]^{2}_{A_{\e}}+\left(\left(\frac{1}{2}-\frac{1}{\theta}  \right)-\frac{1}{k}\right)V_{0} \int_{\R} |u_{n}|^{2}dx \\
&\geq \min\{1, V_{0}\}\left(\left(\frac{1}{2}-\frac{1}{\theta}  \right)-\frac{1}{k}\right) \|u_{n}\|_{\e}^{2}.
\end{align*}
Recalling that $k>\frac{2\theta}{\theta-2}$, we can deduce that $(u_{n})$ is bounded in $\h$. Moreover, by assumption,  $c<\min\{1, V_{0}\} \left(\frac{1}{2}-\frac{1}{\theta} \right)$, so, increasing $k$ if necessary, we can see that
\begin{align}\label{bicchiere}
\limsup_{n\rightarrow \infty}\|u_{n}\|_{\e}^{2}<1.
\end{align}
Since $\h$ is a reflexive space, we can find a subsequence still denoted by $(u_{n})$ and $u\in \h$ such that 
\begin{align}\label{ADOMconv}
\begin{array}{ll}
u_{n}\rightharpoonup u \quad &\mbox{ in } \h \mbox{ as } n\ri \infty, \\
u_{n}\ri u \quad &\mbox{ in } L^{q}_{loc}(\R, \C) \mbox{ for all } q\in [2, \infty) \mbox{ as } n\ri \infty, \\
|u_{n}|\ri |u| \quad &\mbox{ a.e. in } \R \mbox{ as } n\ri \infty.  
\end{array}
\end{align}

By \eqref{bicchiere} and arguing as in Lemma \ref{cor2.3ADOM} we can infer 
\begin{equation}\label{newformula}
\lim_{n\ri \infty} \Re\int_{\R} g(\e x, |u_{n}|^{2})u_{n}\overline{\phi} \, dx = \Re\int_{\R} g_{\e}(x, |u|^{2})u \overline{\phi} \, dx \quad \mbox{ for all } \phi \in C^{\infty}_{c}(\R, \C).
\end{equation}
Taking into account \eqref{ADOMconv}, \eqref{newformula} and the density of $C^{\infty}_{c}(\R, \C)$ in $\h$, we can deduce that
\begin{align*}
\langle J_{\e}'(u), \phi \rangle =0 \quad \mbox{ for all } \phi \in \h, 
\end{align*}
that is $u$ is a critical point for $J_{\e}$.
Consequently, $\langle J_{\e}'(u), u \rangle =0$, or equivalently
\begin{align}\label{new1}
[u]^{2}_{A_{\e}} + \int_{\R} V_{\e}(x) |u|^{2}\, dx = \int_{\R} g_{\e}(x, |u|^{2})|u|^{2}\, dx.
\end{align}
Recalling that $\langle J_{\e}'(u_{n}), u_{n}\rangle =o_{n}(1)$, we also know that
\begin{align}\label{new2}
[u_{n}]^{2}_{A_{\e}} + \int_{\R} V_{\e}(x) |u_{n}|^{2}\, dx = \int_{\R} g_{\e}(x, |u_{n}|^{2})|u_{n}|^{2}\, dx+ o_{n}(1). 
\end{align}
From the compactness of the Sobolev embedding, \eqref{bicchiere} and arguing as in Lemma \ref{cor2.3ADOM} we have
\begin{align}\label{new22}
\lim_{n\ri \infty} \int_{-R}^{R} g_{\e}(x, |u_{n}|^{2})|u_{n}|^{2} dx =  \int_{-R}^{R} g_{\e}(x, |u|^{2})|u|^{2} dx. 
\end{align}
In the light of \eqref{new1}, \eqref{new2} and \eqref{new22}, it will be enough to prove the following claim: \\
{\bf Claim 2} For any $\xi>0$ there exists $R=R_{\xi}>0$ such that $\Lambda_{\e}\subset (-R, R)$ and
\begin{equation}\label{T}
\limsup_{n\rightarrow \infty}\int_{\R\setminus (-R, R)} \int_{\R} \frac{|u_{n}(x)-u_{n}(y)e^{\imath A_{\e}(\frac{x+y}{2})\cdot (x-y)}|^{2}}{|x-y|^{2}} dx dy+\int_{\R\setminus (-R, R)} V_{\e}(x)|u_{n}|^{2}\, dx\leq \xi.
\end{equation}
Indeed, if we assume that \eqref{T} holds true, we can see that $(g_1)$, $(g_2)$ and Sobolev inequality yield
$$
\int_{\R\setminus (-R, R)} g_{\e}(x, |u_{n}|^{2})|u_{n}|^{2}dx<\frac{\xi}{4},
$$
for all $n$ big enough. On the other hand, choosing $R$ large enough, we may assume that
$$
\int_{\R\setminus (-R, R)} g_{\e}(x, |u|^{2})|u|^{2}dx<\frac{\xi}{4},
$$
Combining the above inequalities and \eqref{new22} we can infer that
$$
\lim_{n\ri \infty} \int_{\R} g_{\e}(x, |u_{n}|^{2})|u_{n}|^{2} dx =  \int_{\R} g_{\e}(x, |u|^{2})|u|^{2} dx.
$$
This together with \eqref{new1} and \eqref{new2} yields 
\begin{align*}
\lim_{n\ri \infty} \|u_{n}\|_{\e}^{2} = \|u\|_{\e}^{2}. 
\end{align*}
Since $\h$ is a Hilbert space and $u_{n}\rightharpoonup u$ in $\h$ as $n\ri \infty$, we infer 
$u_{n}\ri u$  in  $\h$ as  $n\ri \infty$. 

Now, it remains to prove the validity of \eqref{T}. Take $\eta_{R}\in C^{\infty}(\R, \R)$ be such that $0\leq \eta_{R}\leq 1$, $\eta_{R}=0$ in $(- \frac{R}{2}, \frac{R}{2})$, $\eta_{R}=1$ in $\R\setminus (-R, R)$ and $|\eta'_{R}|\leq \frac{C}{R}$ for some $C>0$ independent of $R$.

Since $\langle J'_{\e}(u_{n}), \eta_{R}u_{n}\rangle =o_{n}(1)$ we have
\begin{align*}
&\Re\left(\iint_{\R^{2}} \frac{(u_{n}(x)-u_{n}(y)e^{\imath A_{\e}(\frac{x+y}{2})\cdot (x-y)})\overline{(u_{n}(x)\eta_{R}(x)-u_{n}(y)\eta_{R}(y)e^{\imath A_{\e}(\frac{x+y}{2})\cdot (x-y)})}}{|x-y|^{2}}\, dx dy \right)\\
&+\int_{\R} V_{\e}(x)\eta_{R} |u_{n}|^{2}\, dx=\int_{\R} g_{\e}(x, |u_{n}|^{2})|u_{n}|^{2}\eta_{R}\, dx+o_{n}(1).
\end{align*}
Observing that
\begin{align*}
&\Re\left(\iint_{\R^{2}} \frac{(u_{n}(x)-u_{n}(y)e^{\imath A_{\e}(\frac{x+y}{2})\cdot (x-y)})\overline{(u_{n}(x)\eta_{R}(x)-u_{n}(y)\eta_{R}(y)e^{\imath A_{\e}(\frac{x+y}{2})\cdot (x-y)})}}{|x-y|^{2}}\, dx dy \right)\\
&=\Re\left(\iint_{\R^{2}} \overline{u_{n}(y)}e^{-\imath A_{\e}(\frac{x+y}{2})\cdot (x-y)}\frac{(u_{n}(x)-u_{n}(y)e^{\imath A_{\e}(\frac{x+y}{2})\cdot (x-y)})(\eta_{R}(x)-\eta_{R}(y))}{|x-y|^{2}}  \,dx dy\right)\\
&+\iint_{\R^{2}} \eta_{R}(x)\frac{|u_{n}(x)-u_{n}(y)e^{\imath A_{\e}(\frac{x+y}{2})\cdot (x-y)}|^{2}}{|x-y|^{2}}\, dx dy,
\end{align*}
and using $(g_3)$-(ii), we can see that
\begin{align}\label{PS1}
&\iint_{\R^{2}} \eta_{R}(x)\frac{|u_{n}(x)-u_{n}(y)e^{\imath A_{\e}(\frac{x+y}{2})\cdot (x-y)}|^{2}}{|x-y|^{2}}\, dx dy+\int_{\R} V_{\e}(x)\eta_{R} |u_{n}|^{2}\, dx\nonumber\\
&\leq -\Re\left(\iint_{\R^{2}} \overline{u_{n}(y)}e^{-\imath A_{\e}(\frac{x+y}{2})\cdot (x-y)}\frac{(u_{n}(x)-u_{n}(y)e^{\imath A_{\e}(\frac{x+y}{2})\cdot (x-y)})(\eta_{R}(x)-\eta_{R}(y))}{|x-y|^{2}}  \,dx dy\right) \nonumber\\
&+\frac{1}{k}\int_{\R} V_{\e}(x) \eta_{R} |u_{n}|^{2}\, dx+o_{n}(1).
\end{align}
Now, the H\"older inequality and the boundedness of $(u_{n})$ in $\h$ imply
\begin{align}\label{PS2}
&\left|\Re\left(\iint_{\R^{2}} \overline{u_{n}(y)}e^{-\imath A_{\e}(\frac{x+y}{2})\cdot (x-y)}\frac{(u_{n}(x)-u_{n}(y)e^{\imath A_{\e}(\frac{x+y}{2})\cdot (x-y)})(\eta_{R}(x)-\eta_{R}(y))}{|x-y|^{2}}  \,dx dy\right)\right| \nonumber\\
&\leq \left(\iint_{\R^{2}} \frac{|u_{n}(x)-u_{n}(y)e^{\imath A_{\e}(\frac{x+y}{2})\cdot (x-y)}|^{2}}{|x-y|^{2}}\,dxdy  \right)^{\frac{1}{2}} \left(\iint_{\R^{2}} |\overline{u_{n}(y)}|^{2}\frac{|\eta_{R}(x)-\eta_{R}(y)|^{2}}{|x-y|^{2}} \, dxdy\right)^{\frac{1}{2}} \nonumber\\
&\leq C \left(\iint_{\R^{2}} |u_{n}(y)|^{2}\frac{|\eta_{R}(x)-\eta_{R}(y)|^{2}}{|x-y|^{2}} \, dxdy\right)^{\frac{1}{2}}.
\end{align}
Therefore, if we prove that
\begin{equation}\label{PS3}
\lim_{R\rightarrow \infty}\limsup_{n\rightarrow \infty} \iint_{\R^{2}} |u_{n}(y)|^{2}\frac{|\eta_{R}(x)-\eta_{R}(y)|^{2}}{|x-y|^{2}} \, dxdy=0,
\end{equation}
then \eqref{PS1} yields that \eqref{T} holds true. However, the validity of the limit in \eqref{PS3} can be obtained arguing as in the proof of Lemma \ref{lemPSI} observing that in this case we have to use the boundedness of $(u_{n})$ in $\h$ (and then the boundedness of $(|u_{n}|)$ in $L^{q}(\R, \R)$ for all $q\in [2, \infty)$) to get estimates independent of $n\in \mathbb{N}$; see also Lemma $3.4$ in \cite{AI}.
\end{proof}

\noindent
In order to obtain multiple critical points of $J_{\e}$, we will consider $J_{\e}$ constrained on $\N_{\e}$. Therefore, it is needed to prove the following result. 
\begin{prop}\label{propPSc}
Let $c\in \R$ be such that $c<\min\{1, V_{0}\} \left(\frac{1}{2}-\frac{1}{\theta} \right)$. Then, the functional $J_{\e}$ restricted to $\mathcal{N}_{\e}$ satisfies the $(PS)_{c}$ condition at the level $c$.
\end{prop}
\begin{proof}
Let $(u_{n})\subset \mathcal{N}_{\e}$ be such that $J_{\e}(u_{n})\rightarrow c$ and $\|J'_{\e}(u_{n})_{|\mathcal{N}_{\e}}\|_{*}=o_{n}(1)$. Then there exists $(\lambda_{n})\subset \R$ such that
\begin{equation}\label{AFT}
J'_{\e}(u_{n})=\lambda_{n} T'_{\e}(u_{n})+o_{n}(1)
\end{equation}
where $T_{\e}: \h\rightarrow \R$ is given by
\begin{align*}
T_{\e}(u)=\|u\|_{\e}^{2}-\int_{\R} g_{\e}(x, |u|^{2})|u|^{2}\, dx.
\end{align*}
Taking into account $\langle J'_{\e}(u_{n}), u_{n}\rangle=0$, $g_{\e}(x, |u|^{2})$ is constant on $\Lambda_{\e}^{c}\cap \{|u|^{2}>T_{a}\}$, the definition of $g$, the monotonicity of $\xi$ and $(f_6)$, we can see that
\begin{align}
\langle T'_{\e}(u_{n}), u_{n}\rangle&=2\|u_{n}\|_{\e}^{2}-2\int_{\R} g'_{\e}(x, |u_{n}|^{2})|u_{n}|^{4}\, dx-2\int_{\R} g_{\e}(x, |u_{n}|^{2})|u_{n}|^{2}\, dx  \label{22ZS} \\
&=-2\int_{\R^{3}} g'_{\e}(x, |u_{n}|^{2})|u_{n}|^{4}\, dx \nonumber \\
&\leq -2\int_{\Lambda_{\e}\cup \{|u_{n}|^{2}<t_{a}\}} f'(|u_{n}|^{2})|u_{n}|^4 dx \nonumber \\
&\leq -2C_{\sigma} \int_{\Lambda_{\e}\cup \{|u_{n}|^{2}<t_{a}\}} |u_{n}|^{\sigma} dx \nonumber \\
&\leq -2C_{\sigma} \int_{\Lambda_{\e}} |u_{n}|^{\sigma} dx<0.
\end{align}
From the boundedness of $(u_{n})$ in $\h$, we can assume that $\langle T'_{\e}(u_{n}), u_{n}\rangle\rightarrow \ell\leq 0$. 

If $\ell=0$, then  $|u_{n}|\rightarrow 0$ in $L^{\sigma}(\Lambda_{\e}, \R)$. By interpolation, $|u_{n}|\rightarrow 0$ in $L^{r}(\Lambda_{\e}, \R)$ for all $r\geq \sigma$. Now, we note that $u_{n}\in \N_{\e}$ and $J_{\e}(u_{n})\rightarrow c<\min\{1, V_{0}\} \left(\frac{1}{2}-\frac{1}{\theta} \right)$ imply that $\limsup_{n\rightarrow \infty}\|u_{n}\|_{\e}^{2}<1$. Then, using Lemma \ref{cor2.3ADOM} we deduce that 
$$
\int_{\Lambda_{\e}} g_{\e}(x, |u_{n}|^{2})|u_{n}|^{2}\, dx=\int_{\Lambda_{\e}} f(|u_{n}|^{2})|u_{n}|^{2}\, dx\rightarrow 0.
$$
Therefore
$$
\|u_{n}\|^{2}_{\e}=\int_{\Lambda_{\e}^{c}} g_{\e}(x, |u_{n}|^{2})|u_{n}|^{2}\, dx+o_{n}(1)\leq \frac{1}{k}\int_{\Lambda_{\e}^{c}} V_{\e}(x)|u_{n}|^{2} dx+o_{n}(1)
$$
which yields $\|u_{n}\|_{\e}\rightarrow 0$, that is a contradiction because $\|u\|_{\e}\geq r>0$ for all $u\in \N_{\e}$. 

Consequently, $\ell<0$ and taking into account \eqref{AFT} we get $\lambda_{n}\rightarrow 0$, that is $u_{n}$ is a $(PS)_{c}$ sequence for the unconstrained functional.  The thesis follows from Lemma \ref{PSc}.
\end{proof}

\noindent
Arguing as in the previous lemma, it is easy to check that:
\begin{cor}\label{cor}
The critical points of the functional $J_{\e}$ on $\mathcal{N}_{\e}$ are critical points of $J_{\e}$.
\end{cor}

\noindent	
We conclude this section giving the following existence result for \eqref{MPe}:
\begin{thm}
There exists $\e_{0}>0$ such that problem \eqref{MPe} admits a nontrivial solution for all $\e\in (0, \e_{0})$.
\end{thm}
\begin{proof}
Taking into account Lemma \ref{MPG}, Lemma \ref{AMlem1}, Lemma \ref{PSc} and applying mountain pass theorem \cite{AR}, we can deduce that \eqref{MPe} admits a nontrivial solution provided that $\e>0$ is sufficiently small.
\end{proof}

\section{Multiple solutions for the modified problem}

\noindent
In this section, we show that it is possible to relate the number of nontrivial solutions of \eqref{MPe} to the topology of the set $\Lambda$.
For this reason, we consider $\delta>0$ such that
$$
M_{\delta}=\{x\in \R: {\rm dist}(x, M)\leq \delta\}\subset \Lambda,
$$
and we choose $\eta\in C^{\infty}_{0}(\R_{+}, [0, 1])$ such that $\eta(t)=1$ if $0\leq t\leq \frac{\delta}{2}$ and $\eta(t)=0$ if $t\geq \delta$.

For any $y\in \Lambda$, we introduce (see \cite{AD})
$$
\Psi_{\e, y}(x)=\eta(|\e x-y|) w\left(\frac{\e x-y}{\e}\right)e^{\imath \tau_{y} \left( \frac{\e x-y}{\e} \right)},
$$
where $\tau_{y}(x)=\sum_{j=1}^{3}A_{j}(y)x_{j}$ and $w\in H^{1/2}_{V_{0}}(\R, \R)$ is a positive ground state solution to the autonomous problem \eqref{AP0} with $\mu=V_{0}$ (see Lemma \ref{FS}).
Let $t_{\e}>0$ be the unique number such that 
$$
\max_{t\geq 0} J_{\e}(t \Psi_{\e, y})=J_{\e}(t_{\e} \Psi_{\e, y}). 
$$
Noting that $t_{\e} \Psi_{\e, y}\in \N_{\e}$, we can define $\Phi_{\e}: M\rightarrow \N_{\e}$ as
$$
\Phi_{\e}(y)= t_{\e} \Psi_{\e, y}.
$$
\begin{lem}\label{lem3.4}
The functional $\Phi_{\e}$ satisfies the following limit
\begin{equation*}
\lim_{\e\rightarrow 0} J_{\e}(\Phi_{\e}(y))=c_{V_{0}} \quad \mbox{ uniformly in } y\in M.
\end{equation*}
\end{lem}
\begin{proof}
Assume by contradiction that there exist $\delta_{0}>0$, $(y_{n})\subset M$ and $\e_{n}\rightarrow 0$ such that 
\begin{equation}\label{puac}
|J_{\e_{n}}(\Phi_{\e_{n}}(y_{n}))-c_{V_{0}}|\geq \delta_{0}.
\end{equation}
In order to lighten the notation, we write $\Phi_{n}$, $\Psi_{n}$ and $t_{n}$ to denote $\Phi_{\e_{n}}(y_{n})$, $\Psi_{\e_{n}, y_{n}}$ and $t_{\e_{n}}$, respectively.
Arguing as in Lemma $4.1$ in \cite{AD} and applying the Dominated Convergence Theorem we get  
\begin{align}\begin{split}\label{nio3}
\| \Psi_{n} \|^{2}_{\e_{n}}\rightarrow \|w\|^{2}_{V_{0}}\in (0, \infty). 
\end{split}\end{align}
On the other hand, since $\langle J'_{\e_{n}}(t_{n}\Psi_{n}),t_{n}\Psi_{n}\rangle=0$ and using the change of variable $z=\frac{\e_{n}x-y_{n}}{\e_{n}}$ it follows that
\begin{align*}
&t_{n}^{2}\|\Psi_{n}\|_{\e_{n}}^{2} =\int_{\R} g(\e_{n}z+y_{n}, |t_{n}\eta(|\e_{n}z|)w(z)|^{2}) |t_{\e_{n}}\eta(|\e_{n}z|)w(z)|^{2} dz.
\end{align*}
If $z\in (- \frac{\delta}{\e_{n}}, \frac{\delta}{\e_{n}})$ then $\e_{n} z+y_{n}\in (y_{n}-\delta, y_{n}+\delta)\subset M_{\delta}\subset \Lambda$. 
Hence, being $g(x,t)=f(t)$ for all $x\in \Lambda$ and $\eta(t)=0$ for $t\geq \delta$, we have
\begin{align}\label{1nio}
\|\Psi_{n}\|_{\e_{n}}^{2} =\int_{\R} f(|t_{n}\eta(|\e_{n}z|)w(z)|^{2}) |\eta(|\e_{n}z|)w(z)|^{2} dz.
\end{align}
Since $\eta=1$ in $(-\frac{\delta}{2}, \frac{\delta}{2})\subset (-\frac{\delta}{2\e_{n}}, \frac{\delta}{2\e_{n}})$ for all $n$ large enough, we get from \eqref{1nio} and $(f_4)$
\begin{align}\label{nioo}
 \|\Psi_{n}\|_{\e_{n}}^{2}&\geq\int_{-\frac{\delta}{2}}^{\frac{\delta}{2}} f(|t_{n} w(z)|^{2})  |w(z)|^{2}dz \nonumber \\
&\geq  f(|t_{n} \alpha|^{2}) \int_{-\frac{\delta}{2}}^{\frac{\delta}{2}}  |w(z)|^{2}dz,
\end{align}
where
\begin{equation*}
\alpha=\min_{|z|\leq \frac{\delta}{2}} w(z)>0.
\end{equation*} 
Now, if  $t_{n}\rightarrow \infty$, we can use \eqref{nioo}, \eqref{nio3} and $(f_{3})$ to deduce a contradiction.
Therefore $(t_{n})$ is bounded and, up to subsequence, we may assume that $t_{n}\rightarrow t_{0}$ for some $t_{0}\geq 0$.  
Let us prove that $t_{0}>0$. Otherwise, if $t_{0}=0$,
we can use \eqref{nio3}, the growth assumptions on $g$ and \eqref{1nio} to see that
\begin{align*}
\| \Psi_{n}\|_{\e_{n}}^{2}\rightarrow 0
\end{align*}
which is impossible because of $t_{\e} \Psi_{\e, y}\in \N_{\e}$ and \eqref{uNr}. 
Hence $t_{0}>0$.
Taking the limit as $n\rightarrow \infty$ in \eqref{1nio}, we deduce that
\begin{align*}
\|w\|^{2}_{V_{0}}=\int_{\R} f((t_{0} w)^{2}) \,w^{2} \, dx.
\end{align*}
In the light of $w\in \N_{V_{0}}$ and $(f_4)$ we can deduce that $t_{0}=1$. This and the Dominated Convergence Theorem imply that
$$
\int_{\R} F(|t_{n} \Psi_{n}|^{2}) \,dx\ri \int_{\R} F(|w|^{2}) \,dx.
$$
Hence, letting the limit as $n\rightarrow \infty$ in  
$$
 J_{\e_{n}}(\Phi_{n} )=\frac{t_{n}^{2}}{2}\|\Psi_{n}\|_{\e_{n}}^{2}-\frac{1}{2}\int_{\R} F(|t_{n} \Psi_{n}|^{2}) \,dx,
$$
we can conclude that
$$
\lim_{n\rightarrow \infty} J_{\e_{n}}(\Phi_{\e_{n}} (y_{n}))=J_{V_{0}}(w)=c_{V_{0}},
$$
which contradicts \eqref{puac}.
\end{proof}

\noindent 
For any $\delta>0$, we take $\rho=\rho(\delta)>0$ in such way that $M_{\delta}\subset (-\rho, \rho)$. Let $\varUpsilon: \R\ri \R$ be defined as 
\begin{equation*}
\varUpsilon(x)= x \quad \mbox{ if } |x|<\rho \quad \mbox{ and } \quad \varUpsilon(x)=\frac{\rho x}{|x|} \quad \mbox{ if } |x|\geq \rho.
\end{equation*}
Finally, we consider the barycenter map $\beta_{\e}: \N_{\e}\rightarrow \R$ given by
\begin{align*}
\beta_{\e}(u)=\frac{\displaystyle{\int_{\R} \varUpsilon(\e x)|u(x)|^{2} \,dx}}{\displaystyle{\int_{\R} |u(x)|^{2} \,dx}}.
\end{align*}

\noindent
Arguing as Lemma $4.3$ in \cite{AD}, it is easy to see that the function $\beta_{\e}$ verifies the following limit:
\begin{lem}\label{lem3.5N}
\begin{equation*}
\lim_{\e \rightarrow 0} \beta_{\e}(\Phi_{\e}(y))=y \quad \mbox{ uniformly in } y\in M.
\end{equation*}
\end{lem}

The next compactness result is fundamental to show that the solutions of the modified problem are solutions of the original problem.
\begin{lem}\label{prop3.3}
Let $\e_{n}\rightarrow 0$ and $u_{n}\in\N_{\e_{n}}$ for all $n\in \mathbb{N}$ be such that $J_{\e_{n}}(u_{n})\rightarrow c_{V_{0}}$. Then there exists $(\tilde{y}_{n})\subset \R$ such that $v_{n}(x)=|u_{n}|(x+\tilde{y}_{n})$ has a convergent subsequence in $H^{1/2}(\R, \R)$. Moreover, up to a subsequence, $y_{n}=\e_{n} \tilde{y}_{n}\rightarrow y_{0}$ for some $y_{0}\in M$.
\end{lem}
\begin{proof}
Taking into account $\langle J'_{\e_{n}}(u_{n}), u_{n}\rangle=0$, $J_{\e_{n}}(u_{n})= c_{V_{0}}+o_{n}(1)$ and Lemma \ref{FS}, we can argue as in the first part of Lemma \ref{PSc} to see that $(u_{n})$ is bounded in $H^{1/2}_{\e_{n}}$ and $\limsup_{n\rightarrow \infty}\|u_{n}\|_{\e_{n}}^{2}<1$. 
Moreover, from Lemma \ref{DI} and $(V_1)$, we also know that $(|u_{n}|)$ is bounded in $H^{1/2}_{V_{0}}(\R, \R)$.

Now, we prove that there exist a sequence $(\tilde{y}_{n})\subset \R$ and constants $R>0$ and $\gamma>0$ such that
\begin{equation}\label{sacchi}
\liminf_{n\rightarrow \infty}\int_{\tilde{y}_{n}-R}^{\tilde{y}_{n}+R} |u_{n}|^{2} \, dx\geq \gamma>0.
\end{equation}
If by contradiction \eqref{sacchi} does not hold, then for all $R>0$ we get
$$
\lim_{n\rightarrow \infty}\sup_{y\in \R}\int_{y-R}^{y+R} |u_{n}|^{2} \, dx=0.
$$
From the boundedness $(|u_{n}|)$ and Lemma \ref{lions lemma} we can see that $|u_{n}|\rightarrow 0$ in $L^{q}(\R, \R)$ for any $q\in (2, \infty)$. 
Arguing as in Lemma \ref{vanishingF} we  get
\begin{align}\label{glimiti}
\lim_{n\rightarrow \infty}\int_{\R} g_{\e_{n}}(x, |u_{n}|^{2}) |u_{n}|^{2} \,dx=0= \lim_{n\rightarrow \infty}\int_{\R} G_{\e_{n}}(x, |u_{n}|^{2}) \, dx.
\end{align}
Taking into account $\langle J'_{\e_{n}}(u_{n}), u_{n}\rangle=0$ and \eqref{glimiti}, we can  infer that $\|u_{n}\|_{\e_{n}}\rightarrow 0$ as $n\rightarrow \infty$, which implies that  $J_{\e_{n}}(u_{n})\rightarrow 0$, that is a contradiction because $c_{V_{0}}>0$.

Set $v_{n}(x)=|u_{n}|(x+\tilde{y}_{n})$. Then $(v_{n})$ is bounded in $H^{1/2}_{V_{0}}(\R, \R)$, and we may assume that 
$v_{n}\rightharpoonup v\not\equiv 0$ in $H^{1/2}_{V_{0}}(\R, \R)$  as $n\rightarrow \infty$.
Fix $t_{n}>0$ such that $\tilde{v}_{n}=t_{n} v_{n}\in \mathcal{N}_{V_{0}}$. Using Lemma \ref{DI} and $u_{n}\in \mathcal{N}_{\e_{n}}$, we can see that 
$$
c_{V_{0}}\leq J_{V_{0}}(\tilde{v}_{n})\leq \max_{t\geq 0} J_{\e_{n}}(t u_{n})= J_{\e_{n}}(u_{n})= c_{V_{0}}+o_{n}(1)
$$
which implies that $J_{V_{0}}(\tilde{v}_{n})\rightarrow c_{V_{0}}$. In particular, $\tilde{v}_{n}\nrightarrow 0$ in $H^{1/2}_{V_{0}}(\R, \R)$.
Since $(v_{n})$ and $(\tilde{v}_{n})$ are bounded in $H^{1/2}_{V_{0}}(\R, \R)$ and $\tilde{v}_{n}\nrightarrow 0$  in $H^{1/2}_{V_{0}}(\R, \R)$, we deduce that $t_{n}\rightarrow t^{*}\geq 0$. 

Indeed $t^{*}>0$ since $\tilde{v}_{n}\nrightarrow 0$  in $H^{1/2}_{V_{0}}(\R, \R)$. From the uniqueness of the weak limit, we can deduce that $\tilde{v}_{n}\rightharpoonup \tilde{v}=t^{*}v\not\equiv 0$ in $H^{1/2}_{V_{0}}(\R, \R)$. 
This combined with Lemma \ref{FS} yields
\begin{equation}\label{elena}
\tilde{v}_{n}\rightarrow \tilde{v} \quad \mbox{ in } H^{1/2}_{V_{0}}(\R, \R).
\end{equation} 
Consequently, $v_{n}\rightarrow v$ in $H^{1/2}_{V_{0}}(\R, \R)$ as $n\rightarrow \infty$.

Now, we set $y_{n}=\e_{n}\tilde{y}_{n}$ and we show that $(y_{n})$ admits a subsequence, still denoted by $y_{n}$, such that $y_{n}\rightarrow y_{0}$ for some $y_{0}\in M$. Firstly, we prove that $(y_{n})$ is bounded. Assume by contradiction that, up to a subsequence, $|y_{n}|\rightarrow \infty$ as $n\rightarrow \infty$. Take $R>0$ such that $\Lambda \subset (-R, R)$. Since we may suppose that  $|y_{n}|>2R$, we have that for any $z\in (- \frac{R}{\e_{n}}, \frac{R}{\e_{n}})$ 
$$
|\e_{n}z+y_{n}|\geq |y_{n}|-|\e_{n}z|>R.
$$
Now, using $u_{n}\in \N_{\e_{n}}$ for all $n\in \mathbb{N}$, $(V_{1})$, Lemma \ref{DI} and the change of variable $x\mapsto z+\tilde{y}_{n}$ we observe that 
\begin{align}\label{pasq}
[v_{n}]^{2}+\int_{\R} V_{0} v_{n}^{2}\, dx &\leq \int_{\R} g(\e_{n} x+y_{n}, |v_{n}|^{2}) |v_{n}|^{2} \, dx \nonumber\\
&\leq \int_{-\frac{R}{\e_{n}}}^{\frac{R}{\e_{n}}} \tilde{f}(|v_{n}|^{2}) |v_{n}|^{2} \, dx+\int_{\R\setminus (-\frac{R}{\e_{n}}, \frac{R}{\e_{n}} )} f(|v_{n}|^{2}) |v_{n}|^{2} \, dx.
\end{align}
Then, recalling that $v_{n}\rightarrow v$ in $H^{1/2}_{V_{0}}(\R, \R)$ as $n\rightarrow \infty$ and $\tilde{f}(t)\leq \frac{V_{0}}{k}$, we can see that (\ref{pasq}) yields
$$
\min\left\{1, V_{0}\left(1-\frac{1}{k}\right) \right\} \left([v_{n}]^{2}+\int_{\R} |v_{n}|^{2}\, dx\right)=o_{n}(1),
$$
that is $v_{n}\rightarrow 0$ in $H^{1/2}_{V_{0}}(\R, \R)$ and this is a contradiction. Therefore, $(y_{n})$ is bounded and we may assume that $y_{n}\rightarrow y_{0}\in \R$. If $y_{0}\notin \overline{\Lambda}$, then we can argue as above to infer that $v_{n}\rightarrow 0$ in $H^{1/2}_{V_{0}}(\R, \R)$, which is impossible. Hence $y_{0}\in \overline{\Lambda}$. Let us note that if $V(y_{0})=V_{0}$, it follows from $(V_2)$ that $y_{0}\notin \partial \Lambda$. Therefore, it is enough to verify that $V(y_{0})=V_{0}$. Suppose by contradiction that $V(y_{0})>V_{0}$.
Hence, using (\ref{elena}), Fatou's Lemma, the invariance of  $\R$ by translations and Lemma \ref{DI} we get 
\begin{align*}
c_{V_{0}}=J_{V_{0}}(\tilde{v})&<\frac{1}{2}[\tilde{v}]^{2}+\frac{1}{2}\int_{\R} V(y_{0})\tilde{v}^{2} \, dx-\frac{1}{2}\int_{\R} F(|\tilde{v}|^{2})\, dx\\
&\leq \liminf_{n\rightarrow \infty}\left[\frac{1}{2}[\tilde{v}_{n}]^{2}+\frac{1}{2}\int_{\R} V(\e_{n}x+y_{n}) |\tilde{v}_{n}|^{2} \, dx-\frac{1}{2}\int_{\R} F(|\tilde{v}_{n}|^{2})\, dx  \right] \\
&\leq \liminf_{n\rightarrow \infty}\left[\frac{t_{n}^{2}}{2}[|u_{n}|]^{2}+\frac{t_{n}^{2}}{2}\int_{\R} V(\e_{n}z) |u_{n}|^{2} \, dz-\frac{1}{2}\int_{\R} F(|t_{n} u_{n}|^{2})\, dz  \right] \\
&\leq \liminf_{n\rightarrow \infty} J_{\e_{n}}(t_{n} u_{n}) \leq \liminf_{n\rightarrow \infty} J_{\e_{n}}(u_{n})= c_{V_{0}}.
\end{align*}
\end{proof}

Now, we consider the following subset $\widetilde{\N}_{\e}$ of $\N_{\e}$ defined as 
$$
\widetilde{\N}_{\e}=\left \{u\in \N_{\e}: J_{\e}(u)\leq c_{V_{0}}+h(\e)\right\}, 
$$
where $h:\R^{+}\rightarrow \R^{+}$ is such that $h(\e)\rightarrow 0$ as $\e \rightarrow 0$. 
Given $y\in M$, we can use Lemma \ref{lem3.4} to infer that $h(\e)=|J_{\e}(\Phi_{\e}(y))-c_{V_{0}}|\rightarrow 0$ as $\e \rightarrow 0$. Thus $\Phi_{\e}(y)\in \widetilde{\N}_{\e}$, and $\widetilde{\N}_{\e}\neq \emptyset$ for any $\e>0$. Moreover, proceeding as in Lemma $4.5$ in \cite{AD}, we have:
\begin{lem}\label{lem3.5}
For any $\delta>0$ we have
$$
\lim_{\e \rightarrow 0} \sup_{u\in \widetilde{\mathcal{N}}_{\e}} {\rm dist}(\beta_{\e}(u), M_{\delta})=0.
$$
\end{lem}

\noindent
We end this section giving the proof of a multiplicity result for \eqref{MPe}.
\begin{thm}\label{multiple}
For any $\delta>0$ such that $M_{\delta}\subset \Lambda$, there exists $\tilde{\e}_{\delta}>0$ such that, for any $\e\in (0, \tilde{\e}_{\delta})$, problem \eqref{MPe} has at least $cat_{M_{\delta}}(M)$ nontrivial solutions.
\end{thm}
\begin{proof}
Given  $\delta>0$ such that $M_{\delta}\subset \Lambda$, we can use Lemma \ref{lem3.4}, Lemma \ref{lem3.5N}, Lemma \ref{lem3.5} and argue as in \cite{CL} to deduce the existence of $\tilde{\e}_{\delta}>0$ such that, for any $\e\in (0, \e_{\delta})$, the following diagram
$$
M \stackrel{\Phi_{\e}}\rightarrow \widetilde{\mathcal{N}}_{\e} \stackrel{\beta_{\e}}\rightarrow M_{\delta}
$$
is well defined and $\beta_{\e}\circ \Phi_{\e}$ is homotopically equivalent to the embedding  $\iota: M\rightarrow M_{\delta}$. Thus $cat_{\widetilde{\mathcal{N}}_{\e}}(\widetilde{\mathcal{N}}_{\e})\geq cat_{M_{\delta}}(M)$.
It follows from Proposition \ref{propPSc} and standard Ljusternik-Schnirelmann theory \cite{W} that $J_{\e}$  possesses at least $cat_{\widetilde{\mathcal{N}}_{\e}}(\widetilde{\mathcal{N}}_{\e})$  critical points on $\mathcal{N}_{\e}$. Applying Corollary \ref{cor},  we can deduce that  \eqref{MPe} has at least $cat_{M_{\delta}}(M)$ nontrivial solutions.
\end{proof}

\section{Proof of Theorem \ref{thm1}}

\noindent
This last section is devoted to the proof of the main theorem of this work. Indeed, we aim to prove that the solutions obtained in Theorem \ref{multiple} verify $|u_{\e}(x)|\leq t_{a}$ for any $x\in \R\setminus\Lambda_{\e}$ and $\e$ small.

Let us start proving the following lemma which will play a fundamental role in the study of the behavior of the maximum points of the solutions. To do this, we use a Moser iteration argument \cite{Moser} and a sort of Kato's inequality \cite{Kato} for the modulus of solutions to \eqref{MPe}.
\begin{lem}\label{moser} 
Let $\e_{n}\rightarrow 0$ and $u_{n}\in H^{1/2}_{\e_{n}}$ be a solution to \eqref{MPe} such that 
$$
m:=\limsup_{n\ri \infty}\|u_{n}\|^{2}_{\e_{n}}<1.
$$ 
Then $v_{n}=|u_{n}|(\cdot+\tilde{y}_{n})$ satisfies $v_{n}\in L^{\infty}(\R,\R)$ and there exists $C>0$ such that 
$$
\|v_{n}\|_{L^{\infty}(\R)}\leq C \quad \mbox{ for all } n\in \mathbb{N},
$$
where $\tilde{y}_{n}$ is given by Lemma \ref{prop3.3}.
Moreover
$$
\lim_{|x|\rightarrow \infty} v_{n}(x)=0 \mbox{ uniformly in } n\in \mathbb{N}.
$$
\end{lem}
\begin{proof}
For any $L>0$ we define $u_{L,n}:=\min\{|u_{n}|, L\}\geq 0$ and we set $v_{L, n}=u_{L,n}^{2(\beta-1)}u_{n}$ where $\beta>1$ will be chosen later.
Taking $v_{L, n}$ as test function in (\ref{MPe}) we can see that
\begin{align}\label{conto1FF}
&\Re\left(\iint_{\R^{2}} \frac{(u_{n}(x)-u_{n}(y)e^{\imath A(\frac{x+y}{2})\cdot (x-y)})}{|x-y|^{2}} \overline{(u_{n}u_{L,n}^{2(\beta-1)}(x)-u_{n}u_{L,n}^{2(\beta-1)}(y)e^{\imath A(\frac{x+y}{2})\cdot (x-y)})} \, dx dy\right)   \nonumber \\
&=\int_{\R} g_{\e_{n}}(x, |u_{n}|^{2}) |u_{n}|^{2}u_{L,n}^{2(\beta-1)}  \,dx-\int_{\R} V_{\e_{n}}(x) |u_{n}|^{2} u_{L,n}^{2(\beta-1)} \, dx.
\end{align}
Now, we note that
\begin{align*}
&\Re\left[(u_{n}(x)-u_{n}(y)e^{\imath A(\frac{x+y}{2})\cdot (x-y)})\overline{(u_{n}u_{L,n}^{2(\beta-1)}(x)-u_{n}u_{L,n}^{2(\beta-1)}(y)e^{\imath A(\frac{x+y}{2})\cdot (x-y)})}\right] \\
&=\Re\Bigl[|u_{n}(x)|^{2}u_{L,n}^{2(\beta-1)}(x)-u_{n}(x)\overline{u_{n}(y)} u_{L,n}^{2(\beta-1)}(y)e^{-\imath A(\frac{x+y}{2})\cdot (x-y)}-u_{n}(y)\overline{u_{n}(x)} u_{L,n}^{2(\beta-1)}(x) e^{\imath A(\frac{x+y}{2})\cdot (x-y)} \\
&+|u_{n}(y)|^{2}u_{L,n}^{2(\beta-1)}(y) \Bigr] \\
&\geq (|u_{n}(x)|^{2}u_{L,n}^{2(\beta-1)}(x)-|u_{n}(x)||u_{n}(y)|u_{L,n}^{2(\beta-1)}(y)-|u_{n}(y)||u_{n}(x)|u_{L,n}^{2(\beta-1)}(x)+|u_{n}(y)|^{2}u^{2(\beta-1)}_{L,n}(y) \\
&=(|u_{n}(x)|-|u_{n}(y)|)(|u_{n}(x)|u_{L,n}^{2(\beta-1)}(x)-|u_{n}(y)|u_{L,n}^{2(\beta-1)}(y)),
\end{align*}
which yields
\begin{align}\label{realeF}
&\Re\left(\iint_{\R^{2}} \frac{(u_{n}(x)-u_{n}(y)e^{\imath A(\frac{x+y}{2})\cdot (x-y)})}{|x-y|^{2}} \overline{(u_{n}u_{L,n}^{2(\beta-1)}(x)-u_{n}u_{L,n}^{2(\beta-1)}(y)e^{\imath A(\frac{x+y}{2})\cdot (x-y)})} \, dx dy\right) \nonumber\\
&\geq \iint_{\R^{2}} \frac{(|u_{n}(x)|-|u_{n}(y)|)}{|x-y|^{2}} (|u_{n}(x)|u_{L,n}^{2(\beta-1)}(x)-|u_{n}(y)|u_{L,n}^{2(\beta-1)}(y))\, dx dy.
\end{align}
On the other hand, by $(g_1)$ and $(g_2)$, for any $\xi>0$ there exists $C_{\xi}>0$ such that
\begin{equation}\label{SS2}
g_{\e_{n}}(x, |u_{n}|^{2})|u_{n}|^{2}\leq \xi |u_{n}|^{2}+C_{\xi}|u_{n}|^{2} B(u_{n}) \quad \mbox{ for all } n\in \mathbb{N}, 
\end{equation}
where $B(u_{n}):=(e^{\omega \tau |u_{n}|^{2}} -1)$. Arguing as in Lemma \ref{lem2.2ADOM}, we can see that $B(u_{n})\in L^{q}(\R)$ for some $q>1$ close to $1$, with $\tau, q>1$ are such that $\tau q m<1$,  and 
\begin{align}\label{ASJDE}
\|B(u_{n})\|_{L^{q}(\R)}\leq C \quad \mbox{ for any } n\in \mathbb{N}. 
\end{align}
For all $t\geq 0$, let us define
\begin{equation*}
\gamma(t)=\gamma_{L, \beta}(t)=t t_{L}^{2(\beta-1)}
\end{equation*}
where  $t_{L}=\min\{t, L\}$. 
Let us observe that, since $\gamma$ is an increasing function, then it holds
\begin{align*}
(a-b)(\gamma(a)- \gamma(b))\geq 0 \quad \mbox{ for any } a, b\in \R.
\end{align*}
Let 
\begin{equation*}
\Lambda(t)=\frac{|t|^{2}}{2} \quad \mbox{ and } \quad \Gamma(t)=\int_{0}^{t} (\gamma'(\tau))^{\frac{1}{2}} d\tau. 
\end{equation*}
and note that
\begin{equation}\label{Gg}
\Lambda'(a-b)(\gamma(a)-\gamma(b))\geq |\Gamma(a)-\Gamma(b)|^{2} \quad \mbox{ for any } a, b\in\R. 
\end{equation}
In fact, for any $a, b\in \R$ such that $a<b$, the Jensen inequality gives
\begin{align*}
\Lambda'(a-b)(\gamma(a)-\gamma(b))&=(a-b)\int_{b}^{a} \gamma'(t)dt\\
&=(a-b)\int_{b}^{a} (\Gamma'(t))^{2}dt \\
&\geq \left(\int_{b}^{a} \Gamma'(t) dt\right)^{2} \\
&=(\Gamma(a)-\Gamma(b))^{2}.
\end{align*}
Then, \eqref{Gg} yields
\begin{align}\label{Gg1}
|\Gamma(|u_{n}(x)|)- \Gamma(|u_{n}(y)|)|^{2} \leq (|u_{n}(x)|- |u_{n}(y)|)((|u_{n}|u_{L,n}^{2(\beta-1)})(x)- (|u_{n}|u_{L,n}^{2(\beta-1)})(y)). 
\end{align}
In the light of \eqref{realeF} and \eqref{Gg1}, we get
\begin{align}\label{conto1FFF}
\Re\left(\iint_{\R^{2}} \frac{(u_{n}(x)-u_{n}(y)e^{\imath A(\frac{x+y}{2})\cdot (x-y)})}{|x-y|^{2}} \overline{(u_{n}u_{L,n}^{2(\beta-1)}(x)-u_{n}u_{L,n}^{2(\beta-1)}(y)e^{\imath A(\frac{x+y}{2})\cdot (x-y)})} \, dx dy\right) 
\geq [\Gamma(|u_{n}|)]^{2}.
\end{align}
Noting that
$$
\Gamma(|u_{n}|)\geq \frac{1}{\beta} |u_{n}| u_{L,n}^{\beta-1},
$$ 
and recalling that $H^{\frac{1}{2}}(\R, \R)\subset L^{r}(\R, \R)$ for all $r\in [2, \infty)$, we can see that 
\begin{equation}\label{SS1}
[\Gamma(|u_{n}|)]^{2}\geq C \|\Gamma(|u_{n}|)\|^{2}_{L^{\gamma}(\R)}\geq \frac{1}{\beta^{2}} C \||u_{n}| u_{L,n}^{\beta-1}\|^{2}_{L^{\gamma}(\R)}, 
\end{equation}
where $\gamma>2q'$.  

Putting together \eqref{conto1FF}, \eqref{conto1FFF} and \eqref{SS1} we can infer that
\begin{align}\label{BMS}
\left(\frac{1}{\beta}\right)^{2} C\||u_{n}| u_{L,n}^{\beta-1}\|^{2}_{L^{\gamma}(\R)}+\int_{\R} V_{\e_{n}}(x) |u_{n}|^{2}u_{L,n}^{2(\beta-1)} dx\leq \int_{\R} g_{\e_{n}}(x, |u_{n}|^{2}) |u_{n}|^{2} u_{L,n}^{2(\beta-1)} dx.
\end{align}
Taking $\xi\in (0, V_{0})$ and using \eqref{SS2} and \eqref{BMS} we have
\begin{equation}\label{simo1}
\|w_{L,n}\|_{L^{\gamma}(\R)}^{2}\leq C \beta^{2} \int_{\R} B(u_{n}) |u_{n}|^{2}u_{L,n}^{2(\beta-1)}\, dx = C \beta^{2} \int_{\R} B(u_{n}) w_{L, n}^{2}\, dx ,
\end{equation}
where $w_{L,n}:=|u_{n}| u_{L,n}^{\beta-1}$.

From \eqref{ASJDE} and H\"older inequality it follows that
\begin{align*}
\|w_{L,n}\|_{L^{\gamma}(\R)}^{2}\leq C \beta^{2}  \left( \int_{\R} B(u_{n})^{q}\, dx \right)^{\frac{1}{q}} \left( \int_{\R} w_{L, n}^{2q'}\, dx \right)^{\frac{1}{q'}} \leq C \beta^{2} \|w_{L, n}\|_{L^{2q'}(\R)}^{2}, 
\end{align*}
and letting the limit as $L\rightarrow \infty$ we find 
\begin{align*}
\|u_{n}\|_{L^{\beta \gamma}(\R)} \leq C^{\frac{1}{2\beta}}\beta^{\frac{1}{\beta}} \|u_{n}\|_{L^{2q'\beta}(\R)}. 
\end{align*}
Since $\gamma >2q'$, we can use an iteration argument to deduce that for all $m\in \mathbb{N}$
\begin{align*}
\|u_{n}\|_{L^{\tau k^{(m+1)}}(\R)}\leq C^{\sum_{i=1}^{m} k^{-i}} k^{\sum_{i=1}^{m} ik^{-i}} \|u_{n}\|_{L^{\gamma}(\R)} 
\end{align*}
where $k=\frac{\gamma}{2q'}$ and $\tau= 2q'$, which together with the boundedness of $(u_{n})$ in $\h$ implies that 
\begin{equation}\label{UBu}
\|u_{n}\|_{L^{\infty}(\R)}\leq K \quad \mbox{ for all } n\in \mathbb{N}.
\end{equation}

By interpolation, we can also see that $(|u_{n}|)$ strongly converges in $L^{r}(\R, \R)$ for all $r\in (2, \infty)$. In view of the growth assumptions on $g$, also $g_{\e_{n}}(x, |u_{n}|^{2})|u_{n}|$ strongly converges  in the same Lebesgue spaces. 
Now, our claim is to prove that $|u_{n}|$ is a weak subsolution to 
\begin{equation}\label{Kato0}
\left\{
\begin{array}{ll}
(-\Delta)^{1/2}v+V_{\e_{n}}(x) v=g_{\e_{n}}(x, v^{2})v &\mbox{ in } \R \\
v\geq 0 \quad \mbox{ in } \R.
\end{array}
\right.
\end{equation}
Fix $\varphi\in C^{\infty}_{c}(\R, \R)$ such that $\varphi\geq 0$, and we take $\psi_{\delta, n}=\frac{u_{n}}{u_{\delta, n}}\varphi$ as test function in \eqref{MPe}, where we set 
$$
u_{\delta,n}=\sqrt{|u_{n}|^{2}+\delta^{2}} \mbox{ for } \delta>0.
$$ 
We show that $\psi_{\delta, n}\in H^{1/2}_{\e_{n}}$ for all $\delta>0$ and $n\in \mathbb{N}$. Indeed 
$$
\int_{\R} V_{\e_{n}}(x) |\psi_{\delta,n}|^{2} dx\leq \int_{\supp(\varphi)} V_{\e_{n}}(x)\varphi^{2} dx<\infty.
$$  
On the other hand, we can observe
\begin{align*}
\psi_{\delta,n}(x)-\psi_{\delta,n}(y)e^{\imath A_{\e}(\frac{x+y}{2})\cdot (x-y)}&=\left(\frac{u_{n}(x)}{u_{\delta,n}(x)}\right)\varphi(x)-\left(\frac{u_{n}(y)}{u_{\delta,n}(y)}\right)\varphi(y)e^{\imath A_{\e}(\frac{x+y}{2})\cdot (x-y)}\\
&=\left[\left(\frac{u_{n}(x)}{u_{\delta,n}(x)}\right)-\left(\frac{u_{n}(y)}{u_{\delta,n}(x)}\right)e^{\imath A_{\e}(\frac{x+y}{2})\cdot (x-y)}\right]\varphi(x) \\
&+\left[\varphi(x)-\varphi(y)\right] \left(\frac{u_{n}(y)}{u_{\delta,n}(x)}\right) e^{\imath A_{\e}(\frac{x+y}{2})\cdot (x-y)} \\
&+\left(\frac{u_{n}(y)}{u_{\delta,n}(x)}-\frac{u_{n}(y)}{u_{\delta,n}(y)}\right)\varphi(y) e^{\imath A_{\e}(\frac{x+y}{2})\cdot (x-y)}
\end{align*}
which implies that
\begin{align*}
&\left|\psi_{\delta,n}(x)-\psi_{\delta,n}(y)e^{\imath A_{\e}(\frac{x+y}{2})\cdot (x-y)} \right|^{2} \\
&\leq \frac{4}{\delta^{2}} \left|u_{n}(x)-u_{n}(y)e^{\imath A_{\e}(\frac{x+y}{2})\cdot (x-y)} \right|^{2}\|\varphi\|^{2}_{L^{\infty}(\R)} +\frac{4}{\delta^{2}}|\varphi(x)-\varphi(y)|^{2} \|u_{n}\|^{2}_{L^{\infty}(\R)} \\
&+\frac{4}{\delta^{4}} \|u_{n}\|^{2}_{L^{\infty}(\R)} \|\varphi\|^{2}_{L^{\infty}(\R)} |u_{\delta,n}(y)-u_{\delta,n}(x)|^{2} \\
&\leq \frac{4}{\delta^{2}} \left|u_{n}(x)-u_{n}(y)e^{\imath A_{\e}(\frac{x+y}{2})\cdot (x-y)} \right|^{2}\|\varphi\|^{2}_{L^{\infty}(\R)} +\frac{4K^{2}}{\delta^{2}}|\varphi(x)-\varphi(y)|^{2} \\
&+\frac{4K^{2}}{\delta^{4}} \|\varphi\|^{2}_{L^{\infty}(\R)} ||u_{n}(y)|-|u_{n}(x)||^{2}. 
\end{align*}
In the above inequality we used we used the following facts:
$$
|z+w+k|^{2}\leq 4(|z|^{2}+|w|^{2}+|k|^{2}) \quad \mbox{ for all } z,w,k\in \C,
$$ 
$|e^{\imath t}|=1$ for all $t\in \R$, $u_{\delta,n}\geq \delta$, $|\frac{u_{n}}{u_{\delta,n}}|\leq 1$, \eqref{UBu} and 
$$
|\sqrt{|z|^{2}+\delta^{2}}-\sqrt{|w|^{2}+\delta^{2}}|\leq ||z|-|w|| \quad \mbox{ for all } z, w\in \C.
$$
Since $u_{n}\in H^{1/2}_{\e_{n}}$, $|u_{n}|\in H^{1/2}(\R, \R)$ (by Lemma \ref{DI} and $(V_1)$) and $\varphi\in C^{\infty}_{c}(\R, \R)$, we deduce that $\psi_{\delta,n}\in H^{1/2}_{\e_{n}}$.
Accordingly,
\begin{align}\label{Kato1}
&\Re\left[\iint_{\R^{2}} \frac{(u_{n}(x)-u_{n}(y)e^{\imath A_{\e}(\frac{x+y}{2})\cdot (x-y)})}{|x-y|^{2}} \left(\frac{\overline{u_{n}(x)}}{u_{\delta,n}(x)}\varphi(x)-\frac{\overline{u_{n}(y)}}{u_{\delta,n}(y)}\varphi(y)e^{-\imath A_{\e}(\frac{x+y}{2})\cdot (x-y)}  \right) dx dy\right] \nonumber\\
&+\int_{\R} V_{\e_{n}}(x)\frac{|u_{n}|^{2}}{u_{\delta,n}}\varphi \, dx=\int_{\R} g_{\e_{n}}(x, |u_{n}|^{2})\frac{|u_{n}|^{2}}{u_{\delta,n}}\varphi \,dx.
\end{align}
Since $\Re(z)\leq |z|$ for all $z\in \C$ and  $|e^{\imath t}|=1$ for all $t\in \R$, we get
\begin{align}\label{alves1}
&\Re\left[(u_{n}(x)-u_{n}(y)e^{\imath A_{\e}(\frac{x+y}{2})\cdot (x-y)}) \left(\frac{\overline{u_{n}(x)}}{u_{\delta,n}(x)}\varphi(x)-\frac{\overline{u_{n}(y)}}{u_{\delta,n}(y)}\varphi(y)e^{-\imath A_{\e}(\frac{x+y}{2})\cdot (x-y)}  \right)\right] \nonumber\\
&=\Re\left[\frac{|u_{n}(x)|^{2}}{u_{\delta,n}(x)}\varphi(x)+\frac{|u_{n}(y)|^{2}}{u_{\delta,n}(y)}\varphi(y)-\frac{u_{n}(x)\overline{u_{n}(y)}}{u_{\delta,n}(y)}\varphi(y)e^{-\imath A_{\e}(\frac{x+y}{2})\cdot (x-y)} -\frac{u_{n}(y)\overline{u_{n}(x)}}{u_{\delta,n}(x)}\varphi(x)e^{\imath A_{\e}(\frac{x+y}{2})\cdot (x-y)}\right] \nonumber \\
&\geq \left[\frac{|u_{n}(x)|^{2}}{u_{\delta,n}(x)}\varphi(x)+\frac{|u_{n}(y)|^{2}}{u_{\delta,n}(y)}\varphi(y)-|u_{n}(x)|\frac{|u_{n}(y)|}{u_{\delta,n}(y)}\varphi(y)-|u_{n}(y)|\frac{|u_{n}(x)|}{u_{\delta,n}(x)}\varphi(x) \right].
\end{align}
Let us note that
\begin{align}\label{alves2}
&\frac{|u_{n}(x)|^{2}}{u_{\delta,n}(x)}\varphi(x)+\frac{|u_{n}(y)|^{2}}{u_{\delta,n}(y)}\varphi(y)-|u_{n}(x)|\frac{|u_{n}(y)|}{u_{\delta,n}(y)}\varphi(y)-|u_{n}(y)|\frac{|u_{n}(x)|}{u_{\delta,n}(x)}\varphi(x) \nonumber\\
&=  \frac{|u_{n}(x)|}{u_{\delta,n}(x)}(|u_{n}(x)|-|u_{n}(y)|)\varphi(x)-\frac{|u_{n}(y)|}{u_{\delta,n}(y)}(|u_{n}(x)|-|u_{n}(y)|)\varphi(y) \nonumber\\
&=\left[\frac{|u_{n}(x)|}{u_{\delta,n}(x)}(|u_{n}(x)|-|u_{n}(y)|)\varphi(x)-\frac{|u_{n}(x)|}{u_{\delta,n}(x)}(|u_{n}(x)|-|u_{n}(y)|)\varphi(y)\right] \nonumber\\
&+\left(\frac{|u_{n}(x)|}{u_{\delta,n}(x)}-\frac{|u_{n}(y)|}{u_{\delta,n}(y)} \right) (|u_{n}(x)|-|u_{n}(y)|)\varphi(y) \nonumber\\
&=\frac{|u_{n}(x)|}{u_{\delta,n}(x)}(|u_{n}(x)|-|u_{n}(y)|)(\varphi(x)-\varphi(y)) +\left(\frac{|u_{n}(x)|}{u_{\delta,n}(x)}-\frac{|u_{n}(y)|}{u_{\delta,n}(y)} \right) (|u_{n}(x)|-|u_{n}(y)|)\varphi(y) \nonumber\\
&\geq \frac{|u_{n}(x)|}{u_{\delta,n}(x)}(|u_{n}(x)|-|u_{n}(y)|)(\varphi(x)-\varphi(y)) 
\end{align}
where in the last inequality we used the fact that
$$
\left(\frac{|u_{n}(x)|}{u_{\delta,n}(x)}-\frac{|u_{n}(y)|}{u_{\delta,n}(y)} \right) (|u_{n}(x)|-|u_{n}(y)|)\varphi(y)\geq 0
$$
and that
$$
h(t)=\frac{t}{\sqrt{t^{2}+\delta^{2}}} \quad \mbox{ is increasing for } t\geq 0 \mbox{ and }  \varphi\geq 0 \mbox{ in }\R.
$$
Observing that
$$
\frac{\left|\frac{|u_{n}(x)|}{u_{\delta,n}(x)}(|u_{n}(x)|-|u_{n}(y)|)(\varphi(x)-\varphi(y))\right|}{|x-y|^{2}}\leq \frac{||u_{n}(x)|-|u_{n}(y)||}{|x-y|} \frac{|\varphi(x)-\varphi(y)|}{|x-y|}\in L^{1}(\R^{2}),
$$
and 
$$
\frac{|u_{n}(x)|}{u_{\delta,n}(x)}\rightarrow 1 \quad \mbox{ a.e. in } \R \mbox{ as } \delta\rightarrow 0,
$$
we can use \eqref{alves1}, \eqref{alves2} and the Dominated Convergence Theorem to deduce that
\begin{align}\label{Kato2}
&\limsup_{\delta\rightarrow 0} \Re\left[\iint_{\R^{2}} \frac{(u_{n}(x)-u_{n}(y)e^{\imath A_{\e}(\frac{x+y}{2})\cdot (x-y)})}{|x-y|^{2}} \left(\frac{\overline{u_{n}(x)}}{u_{\delta,n}(x)}\varphi(x)-\frac{\overline{u_{n}(y)}}{u_{\delta,n}(y)}\varphi(y)e^{-\imath A_{\e}(\frac{x+y}{2})\cdot (x-y)}  \right) dx dy\right] \nonumber\\
&\geq \limsup_{\delta\rightarrow 0} \iint_{\R^{2}} \frac{|u_{n}(x)|}{u_{\delta,n}(x)}(|u_{n}(x)|-|u_{n}(y)|)(\varphi(x)-\varphi(y)) \frac{dx dy}{|x-y|^{2}} \nonumber\\
&=\iint_{\R^{2}} \frac{(|u_{n}(x)|-|u_{n}(y)|)(\varphi(x)-\varphi(y))}{|x-y|^{2}} dx dy.
\end{align}
Invoking the Dominated Convergence Theorem again (we recall that $\frac{|u_{n}|^{2}}{u_{\delta, n}}\leq |u_{n}|$ and $\varphi\in C^{\infty}_{c}(\R, \R)$) we obtain
\begin{equation}\label{Kato3}
\lim_{\delta\rightarrow 0} \int_{\R} V_{\e_{n}}(x)\frac{|u_{n}|^{2}}{u_{\delta,n}}\varphi \, dx=\int_{\R} V_{\e_{n}}(x) |u_{n}|\varphi \, dx
\end{equation}
and
\begin{equation}\label{Kato4}
\lim_{\delta\rightarrow 0}  \int_{\R} g_{\e_{n}}(x, |u_{n}|^{2})\frac{|u_{n}|^{2}}{u_{\delta,n}}\varphi \, dx=\int_{\R} g_{\e_{n}}(x, |u_{n}|^{2}) |u_{n}|\varphi \, dx.
\end{equation}
Putting together \eqref{Kato1}, \eqref{Kato2}, \eqref{Kato3} and \eqref{Kato4} we can deduce that
\begin{align*}
\iint_{\R^{2}} \frac{(|u_{n}(x)|-|u_{n}(y)|)(\varphi(x)-\varphi(y))}{|x-y|^{2}} dx dy+\int_{\R} V_{\e_{n}}(x) |u_{n}|\varphi \, dx\leq 
\int_{\R} g_{\e_{n}}(x, |u_{n}|^{2}) |u_{n}|\varphi \, dx
\end{align*}
for any $\varphi\in C^{\infty}_{c}(\R, \R)$ such that $\varphi\geq 0$, that is $|u_{n}|$ is a weak subsolution to \eqref{Kato0}.\\
Now, using $(V_{1})$, we can note that $v_{n}=|u_{n}|(\cdot+\tilde{y}_{n})$ solves 
\begin{equation}\label{Pkat}
(-\Delta)^{1/2} v_{n} + V_{0}v_{n}\leq g(\e_{n} x+\e_{n}\tilde{y}_{n}, v_{n}^{2})v_{n} \quad \mbox{ in } \R. 
\end{equation}
Let $z_{n}\in H^{\frac{1}{2}}(\R, \R)$ be a solution to
\begin{equation}\label{US}
(-\Delta)^{1/2} z_{n} + V_{0}z_{n}=g_{n} \quad \mbox{ in } \R,
\end{equation}
where
$$
g_{n}:=g(\e_{n} x+\e_{n}\tilde{y}_{n}, v_{n}^{2})v_{n}\in L^{r}(\R, \R) \quad \mbox{ for any } r\in [2, \infty].
$$
Since \eqref{UBu} yields $\|v_{n}\|_{L^{\infty}(\R)}\leq C$ for all $n\in \mathbb{N}$, by interpolation we know that $v_{n}\rightarrow v$ strongly converges in $L^{r}(\R, \R)$ for all $r\in (2, \infty)$, for some $v\in L^{r}(\R, \R)$, and from the growth assumptions on $f$, we can also see that $g_{n}\rightarrow  f(v^{2})v$ in $L^{r}(\R, \R)$ and $\|g_{n}\|_{L^{\infty}(\R)}\leq C$ for all $n\in \mathbb{N}$.
Then, we deduce that $z_{n}=\mathcal{K}*g_{n}$, where $\mathcal{K}$ is the Bessel kernel (see \cite{BG, FQT}), and arguing as in Lemma $3.12$ in \cite{ADOM}, we deduce that $|z_{n}(x)|\rightarrow 0$ as $|x|\rightarrow \infty$ uniformly in $n\in \mathbb{N}$.
Since $v_{n}$ satisfies \eqref{Pkat} and $z_{n}$ solves \eqref{US}, a comparison argument gives $0\leq v_{n}\leq z_{n}$ a.e. in $\R$ and for all $n\in \mathbb{N}$. Furthermore, we can infer that $v_{n}(x)\rightarrow 0$ as $|x|\rightarrow \infty$ uniformly in $n\in \mathbb{N}$.
\end{proof}

\noindent
Now, we are ready to present the proof of the main result of this paper. 
\begin{proof}[Proof of Theorem \ref{thm1}]
Let $\delta>0$ be such that $M_{\delta}\subset \Lambda$, and we show that there exists  $\hat{\e}_{\delta}>0$ such that for any $\e\in (0, \hat{\e}_{\delta})$ and any solution $u\in \widetilde{\mathcal{N}}_{\e}$ of \eqref{MPe} we have
\begin{equation}\label{Ua}
\|u\|_{L^{\infty}(\R\setminus \Lambda_{\e})}<t_{a}.
\end{equation}
Assume by contradiction that for some sequence $\e_{n}\rightarrow 0$ we can obtain  $u_{n}\in \widetilde{\mathcal{N}}_{\e_{n}}$ verifying $J'_{\e_{n}}(u_{n})=0$ and
\begin{equation}\label{AbsAFF}
\|u_{n}\|_{L^{\infty}(\R\setminus \Lambda_{\e})}\geq t_{a}.
\end{equation}
Since $J_{\e_{n}}(u_{n})\leq c_{V_{0}}+h_{1}(\e_{n})$, we can argue as in the first part of Lemma \ref{prop3.3} to see that $J_{\e_{n}}(u_{n})\rightarrow c_{V_{0}}$. This fact together with $J'_{\e_{n}}(u_{n})=0$ implies that 
$$
\limsup_{n\rightarrow \infty}\|u_{n}\|_{\e_{n}}^{2}<1.
$$ 
Using Lemma \ref{prop3.3} there exists $(\tilde{y}_{n})\subset \R^{3}$ such that $\e_{n}\tilde{y}_{n}\rightarrow y_{0}$ for some $y_{0} \in M$. 
Now, we can find $r>0$ such that, for some subsequence still denoted by itself, we obtain $(\tilde{y}_{n}-r, \tilde{y}_{n}+r)\subset \Lambda$ for all $n\in \mathbb{N}$.
Hence, $(\tilde{y}_{n}-\frac{r}{\e_{n}}, \tilde{y}_{n}+\frac{r}{\e_{n}})\subset \Lambda_{\e_{n}}$ for all $n\in \mathbb{N}$, which implies that
$$
\R\setminus \Lambda_{\e_{n}}\subset \R \setminus \left(\tilde{y}_{n}-\frac{r}{\e_{n}}, \tilde{y}_{n}+\frac{r}{\e_{n}} \right) \quad \mbox{ for any } n\in \mathbb{N}.
$$ 
Invoking Lemma \ref{moser}, there exists $R>0$ such that 
$$
v_{n}(x)<t_{a} \quad \mbox{ for } |x|\geq R, n\in \mathbb{N},
$$ 
where $v_{n}(x)=|u_{n}|(x+ \tilde{y}_{n})$. 
Hence $|u_{n}(x)|<t_{a}$ for any $x\in \R\setminus \left(\tilde{y}_{n}-R, \tilde{y}_{n}+R\right)$ and $n\in \mathbb{N}$. Then we  can find $\nu \in \mathbb{N}$ such that for any $n\geq \nu$ and $r/\e_{n}>R$ it holds 
$$
\R\setminus \Lambda_{\e_{n}}\subset \R \setminus \left(\tilde{y}_{n}- \frac{r}{\e_{n}}, \tilde{y}_{n}+\frac{r}{\e_{n}}\right)\subset \R\setminus (\tilde{y}_{n}-R, \tilde{y}_{n}+R).
$$ 
Then $|u_{n}(x)|<t_{a}$ for any $x\in \R\setminus \Lambda_{\e_{n}}$ and $n\geq \nu$, and this contradicts \eqref{AbsAFF}.

Let $\tilde{\e}_{\delta}>0$ be given by Theorem \ref{multiple} and we set $\e_{\delta}=\min\{\tilde{\e}_{\delta}, \hat{\e}_{\delta} \}$. Applying Theorem \ref{multiple} we obtain $cat_{M_{\delta}}(M)$ nontrivial solutions to \eqref{MPe}.
If $u\in \h$ is one of these solutions, then $u\in \widetilde{\mathcal{N}}_{\e}$, and in view of \eqref{Ua} and the definition of $g$, we can infer that $u$ is also a solution to \eqref{MPe}. Since $\hat{u}_{\e}(x)=u_{\e}(x/\e)$ is a solution to (\ref{P}), we can deduce that \eqref{P} has at least $cat_{M_{\delta}}(M)$ nontrivial solutions.

Finally, we study the behavior of the maximum points of  $|\hat{u}_{n}|$. Take $\e_{n}\rightarrow 0$ and $(u_{n})$ a sequence of solutions to \eqref{MPe} as above. We first note that $(g_1)$ implies that there exists $\gamma> 0$ such that
\begin{align}\label{4.18HZ}
g_{\e}(x, t^{2})t^{2}\leq \frac{V_{0}}{2}t^{2} \quad \mbox{ for all } x\in \R, |t|\leq \gamma.
\end{align}
Arguing as above, we can take $R>0$ such that
\begin{align}\label{4.19HZ}
\|u_{n}\|_{L^{\infty}(\R\setminus (\tilde{y}_{n}-R, \tilde{y}_{n}+R))}<\gamma.
\end{align}
Up to a subsequence, we may also assume that
\begin{align}\label{4.20HZ}
\|u_{n}\|_{L^{\infty}(\tilde{y}_{n}-R, \tilde{y}_{n}+R)}\geq \gamma.
\end{align}

Indeed, if \eqref{4.20HZ} is not true, we get $\|u_{n}\|_{L^{\infty}(\R)}< \gamma$, and it follows from $J_{\e_{n}}'(u_{n})=0$, \eqref{4.18HZ} and Lemma \ref{DI} that 
$$
[|u_{n}|]^{2}+\int_{\R}V_{0}|u_{n}|^{2}dx\leq \|u_{n}\|^{2}_{\e_{n}}=\int_{\R} g_{\e_{n}}(x, |u_{n}|^{2})|u_{n}|^{2}\,dx\leq \frac{V_{0}}{2}\int_{\R}|u_{n}|^{2}\, dx
$$
which gives $\||u_{n}|\|_{V_{0}}=0$ that is a contradiction. Hence \eqref{4.20HZ} holds.

Taking into account \eqref{4.19HZ} and \eqref{4.20HZ}, we can infer that the maximum points $p_{n}$ of $|u_{n}|$ belong to $(\tilde{y}_{n}-R, \tilde{y}_{n}+R)$, that is $p_{n}=\tilde{y}_{n}+q_{n}$ for some $q_{n}\in (-R, R)$. Recalling that the associated solution of \eqref{P} is of the form $\hat{u}_{n}(x)=u_{n}(x/\e_{n})$, we can see that a maximum point $\eta_{n}$ of $|\hat{u}_{n}|$ is $\eta_{n}=\e_{n}\tilde{y}_{n}+\e_{n}q_{n}$. Since $q_{n}\in (-R, R)$, $\e_{n}\tilde{y}_{n}\rightarrow y_{0}$ and $V(y_{0})=V_{0}$, from the continuity of $V$ we can conclude that
$$
\lim_{n\rightarrow \infty} V(\eta_{n})=V_{0}.
$$

Next we estimate the decay properties of $|\hat{u}_{n}|$.
Arguing as in Lemma $4.3$ in \cite{FQT}, we can find a function $w$ such that 
\begin{align}\label{HZ1}
0<w(x)\leq \frac{C}{1+|x|^{2}},
\end{align}
and
\begin{align}\label{HZ2}
(-\Delta)^{1/2} w+\frac{V_{0}}{2}w\geq 0 \quad \mbox{ in } \R\setminus (-R_{1}, R_{1}) 
\end{align}
for some suitable $R_{1}>0$. Using Lemma \ref{moser}, we know that $v_{n}(x)\rightarrow 0$ as $|x|\rightarrow \infty$ uniformly in $n\in \mathbb{N}$, so there exists $R_{2}>0$ such that
\begin{equation}\label{hzero}
h_{n}=g(\e_{n}x+\e_{n}\tilde{y}_{n}, v_{n}^{2})v_{n}\leq \frac{V_{0}}{2}v_{n}  \quad \mbox{ in } \R \setminus (-R_{2}, R_{2}).
\end{equation}
Let $w_{n}$ be a solution to 
$$
(-\Delta)^{1/2}w_{n}+V_{0}w_{n}=h_{n} \quad \mbox{ in } \R.
$$
Then $w_{n}(x)\rightarrow 0$ as $|x|\rightarrow \infty$ uniformly in $n\in \mathbb{N}$, and by comparison $0\leq v_{n}\leq w_{n}$ in $\R$. Moreover, due to \eqref{hzero}, it holds
\begin{align}\label{HZ3}
(-\Delta)^{1/2}w_{n}+\frac{V_{0}}{2}w_{n}=h_{n}-\frac{V_{0}}{2}w_{n}\leq 0 \quad \mbox{ in } \R \setminus (-R_{2}, R_{2}).
\end{align}
Choose $R_{3}=\max\{R_{1}, R_{2}\}$ and we set 
\begin{align}\label{HZ4}
d=\inf_{(-R_{3}, R_{3})} w>0 \quad \mbox{ and } \quad \tilde{w}_{n}=(b+1)w-d\, w_{n}.
\end{align}
where $b=\sup_{n\in \mathbb{N}} \|w_{n}\|_{L^{\infty}(\R)}<\infty$. 
Next, we show that 
\begin{equation}\label{HZ5}
\tilde{w}_{n}\geq 0 \quad \mbox{ in } \R.
\end{equation}
For this purpose, we assume by contradiction that there exists a sequence $(\bar{x}_{j, n})\subset \R$ such that 
\begin{align}\label{HZ6}
\inf_{x\in \R} \tilde{w}_{n}(x)=\lim_{j\rightarrow \infty} \tilde{w}_{n}(\bar{x}_{j, n})<0. 
\end{align}
Notice that
\begin{align} \label{HZ0N}
\lim_{|x|\rightarrow \infty} \sup_{n\in \mathbb{N}}\tilde{w}_{n}(x)=0,  
\end{align}
so that $(\bar{x}_{j, n})$ is bounded, and, up to subsequence, we may assume that there exists $\bar{x}_{n}\in \R$ such that $\bar{x}_{j, n}\rightarrow \bar{x}_{n}$ as $j\rightarrow \infty$. 
Thus, (\ref{HZ6}) becomes
\begin{align}\label{HZ7}
\inf_{x\in \R} \tilde{w}_{n}(x)= \tilde{w}_{n}(\bar{x}_{n})<0.
\end{align}
Using the minimality of $\bar{x}_{n}$ and the representation formula for the fractional Laplacian \cite{DPV}, we can see that 
\begin{align}\label{HZ8}
(-\Delta)^{1/2}\tilde{w}_{n}(\bar{x}_{n})=\frac{1}{2\pi} \int_{\R} \frac{2\tilde{w}_{n}(\bar{x}_{n})-\tilde{w}_{n}(\bar{x}_{n}+\xi)-\tilde{w}_{n}(\bar{x}_{n}-\xi)}{|\xi|^{2}} \, d\xi\leq 0.
\end{align}
From \eqref{HZ4}, we get
\begin{align}\label{HZ0}
\tilde{w}_{n}\geq bd+w-b\, d>0 \quad \mbox{ in } (-R_{3}, R_{3}),
\end{align}
In the light of (\ref{HZ6}) and (\ref{HZ0}), we obtain that $\bar{x}_{n}\in \R\setminus (-R_{3}, R_{3})$.
From \eqref{HZ1}, \eqref{HZ2} and \eqref{HZ3} we can infer that
\begin{align}\label{HZ00}
(-\Delta)^{1/2} \tilde{w}_{n}+\frac{V_{0}}{2}\tilde{w}_{n}\geq 0 \quad \mbox{ in } \R\setminus (-R_{3}, R_{3}).
\end{align}
On the other hand, thanks to (\ref{HZ7}) and (\ref{HZ8}), we get 
$$
(-\Delta)^{1/2} \tilde{w}_{n}(\bar{x}_{n})+\frac{V_{0}}{2}\tilde{w}_{n}(\bar{x}_{n})<0,
$$
which contradicts (\ref{HZ00}).
Accordingly, (\ref{HZ5}) holds true, and using (\ref{HZ1}) and $v_{n}\leq w_{n}$ we get
\begin{align*}
0\leq v_{n}(x)\leq w_{n}(x)\leq \frac{(b+1)}{d}w(x)\leq \frac{\tilde{C}}{1+|x|^{2}} \quad \mbox{ for all } n\in \mathbb{N}, x\in \R,
\end{align*}
for some constant $\tilde{C}>0$. 
Recalling the definition of $v_{n}$, we have  
\begin{align*}
|\hat{u}_{n}|(x)&=|u_{n}|\left(\frac{x}{\e_{n}}\right)=v_{n}\left(\frac{x}{\e_{n}}-\tilde{y}_{n}\right) \\
&\leq \frac{\tilde{C}}{1+|\frac{x}{\e_{n}}-\tilde{y}_{n}|^{2}} \\
&=\frac{\tilde{C} \e_{n}^{2}}{\e_{n}^{2}+|x- \e_{n} \tilde{y}_{n}|^{2}} \\
&\leq \frac{\tilde{C} \e_{n}^{2}}{\e_{n}^{2}+|x-\eta_{n}|^{2}} \quad \forall x\in \R.
\end{align*}
Therefore, the proof of Theorem \ref{thm1} is complete.
\end{proof}

\noindent {\bf Acknowledgements.} 
The author warmly thanks the anonymous referee for her/his useful and nice comments on the paper.

\end{document}